\documentclass[11pt]{amsart}
\usepackage[hmargin=2.6cm,bmargin=2.55cm,tmargin=3.05cm]{geometry}
\usepackage[usenames]{color}
\usepackage{graphbox}
\usepackage{subcaption}

\usepackage{mathtools}
\numberwithin{equation}{section}
\mathtoolsset{showonlyrefs,showmanualtags}

\allowdisplaybreaks

\usepackage{amsmath,amssymb,amsthm,amsfonts,enumerate,url,mathbbol,mathrsfs,esint,tikz,graphicx,pifont,float}
\usetikzlibrary{calc}
\usepackage{hyperref}

\usepackage[utf8]{inputenc}

\numberwithin{equation}{section}

\theoremstyle{plain}
\newtheorem{theorem}{Theorem}[section]
\newtheorem{lemma}[theorem]{Lemma}
\newtheorem{proposition}[theorem]{Proposition}
\newtheorem{corollary}[theorem]{Corollary}

\theoremstyle{definition}
\newtheorem{definition}[theorem]{Definition}

\newtheorem{remark}[theorem]{Remark}
\newtheorem{problem}[theorem]{Open Problem}

\numberwithin{equation}{section}
\numberwithin{figure}{section}

\newcommand{\R}{\mathbb{R}}
\newcommand{\N}{\mathbb{N}}
\newcommand{\Z}{\mathbb{Z}}

\newcommand{\PS}{\mathcal{S}}

\newcommand{\GammaClass}{\mathcal{G}_\Gamma}
\newcommand{\GClass}{\mathcal{G}_G}

\newcommand{\gamaste}{\Gamma^\ast}

\title{Bounds on eigenvalue ratios of quantum graph Laplacians}
\author{Evans M. Harrell II}
\author{James B.\ Kennedy}
\author{Gabriel J. Ramos}

\address{School of Mathematics, Georgia Institute of Technology, Atlanta, GA
30332-0160, United States of America}
\email{harrell@math.gatech.edu}

\address{James B. Kennedy, Departament of Mathematics, University of Aveiro, 3810-193 Aveiro, Portugal}
\email{jbkennedy@ua.pt}

\address{Grupo de F\'isica Matem\'atica, Instituto Superior T\'ecnico, Av.\ Rovisco Pais, 1049-001 Lisboa, Portugal} 
\email{fc58288@alunos.ciencias.ulisboa.pt}



\thanks{This work was supported by the Funda\c{c}\~ao para a Ci\^encia e a Tecnologia (FCT, \url{https://ror.org/00snfqn58}), Portugal, within the scope of the project Spectral Optimal Partitions: geometric and numerical analysis, reference 2023.13921.PEX, \url{https://doi.org/10.54499/2023.13921.PEX} (JBK), and via the research centers CIDMA, references UID/04106/2025, \url{https://doi.org/10.54499/UID/04106/2025} and UID/PRR/04106/2025, \url{https://doi.org/10.54499/UID/PRR/04106/2025} (JBK) and GFM, references UID/00208/2025 \url{https://doi.org/10.54499/UID/00208/2025} (all authors) and UID/PRR/00208/2025 \url{https://doi.org/10.54499/UID/PRR/00208/2025} (JBK) under the FCT Multi-Annual Financing Program for R{\&}D Units.
}
\date{\today}

\begin{document}

\begin{abstract}
    We study ratios of eigenvalues of the Laplacian on compact metric graphs. Our goals are threefold: First, we prove a sharp Ashbaugh--Benguria-type bound for the ratio of the first two eigenvalues on compact trees with Dirichlet conditions at all leaves, concretely showing that the ratio is maximized when the graph is an interval or an equilateral star. This improves a previous Payne--P\'olya--Weinberger-type result due to Nicaise [Bull.\ Sci.\ Math., II.\ S\'er.\ \textbf{111} (1987), 401--413]. Second, we extend this bound to a set of inequalities for the ratio of any pair of eigenvalues of such compact Dirichlet trees which respect the Weyl asymptotics up to an absolute constant. Third, we show that on non-trees, on which we also allow any mix of Neumann and Dirichlet conditions at the leaves, it is possible to recover bounds on the eigenvalue ratios depending only on the number of independent cycles and the number of Neumann leaves, in addition to the eigenvalue indices. This complements previously known counterexamples to analogues of the Ashbaugh--Benguria bound for general quantum graphs, by showing that the only way the bound can fail is through cycles and Neumann leaves, and by explicitly quantifying the extent to which it can fail.
\end{abstract}

\maketitle

\section{Introduction}
\label{sec:intro}

Our goal is to give an (almost) complete treatment of estimates for ratios of eigenvalues of Laplacians on compact metric graphs.

Over 30 years ago, Ashbaugh and Benguria \cite{AB92,AB91} proved that the first two eigenvalues $0 < \lambda_1 (\Omega) < \lambda_2 (\Omega)$ of the Dirichlet Laplacian on a bounded Euclidean domain $\Omega \subset \R^d$ satisfy the sharp, scale-invariant bound
\begin{equation}
\label{eq:ab-domain}
    \frac{\lambda_2 (\Omega)}{\lambda_1 (\Omega)} \leq \frac{\lambda_2(B)}{\lambda_1(B)} = \frac{j_{\frac{d}{2},1}^2}{j_{\frac{d}{2}-1,1}^2},
\end{equation}
where $B$ is any ball in $\R^d$, thus proving an old conjecture of Payne, P\'olya and Weinberger (PPW) \cite{PPW56}, who had established the weaker inequality
\begin{equation}
\label{eq:ppw-domain}
    \frac{\lambda_2 (\Omega)}{\lambda_1(\Omega)} \leq 1 + \frac{4}{d}.
\end{equation}
This was one of the starting points of the study of \emph{universal inequalities} for eigenvalues, that is, inequalities where the bound depends 
only
on a dimensional constant. See \cite{A17} for an excellent and still fairly current survey.

On quantum graphs, the situation is more complicated. Let $\Gamma$ be a compact metric graph and let $0 < \lambda_1 (\Gamma) \leq \lambda_2 (\Gamma)$ be the first two eigenvalues of the Laplacian with some mix of Dirichlet and 
standard, or Neumann--Kirchhoff, conditions at the vertices (see Section~\ref{sec:prelim} for details; for brevity we will simply use the term ``Kirchhoff conditions''). For the meantime we assume at least one Dirichlet vertex to ensure that $\lambda_1 (\Gamma) \neq 0$).

Nicaise, in his seminal paper \cite{N87}, already extended the PPW argument to the case of \emph{Dirichlet trees}, that is, cases where the graph $\Gamma$ has no cycles and all degree-one vertices are equipped with a Dirichlet condition (``Dirichlet leaves''), but all interior vertices satisfy a Kirchhoff condition.  
In \cite[Th\'eor\`eme 4.3]{N87}, Nicaise
actually proved PPW for any pair of consecutive eigenvalues, and a slightly stronger inequality for the first two:

\begin{theorem}[Nicaise]
\label{thm:nicaise}
    Let $\Gamma$ be a compact Dirichlet tree, and let $k \geq 1$, then
    \begin{displaymath}
        \frac{\lambda_{k+1} (\Gamma)}{\lambda_k (\Gamma)} \leq 5.
    \end{displaymath}
    If $k=1$, the constant may be improved to $2+\sqrt{5} \approx 4.236$.
\end{theorem}

This is not the whole story, however. It was shown in \cite{DH10}, see also \cite[Section~6]{BHY24}, that if we allow $\Gamma$ to have cycles (and/or Neumann conditions at degree-one vertices, ``Neumann leaves''), then not only does the upper bound in Theorem~\ref{thm:nicaise} fail, but \emph{no} universal upper bound is possible. A prototypical counterexample sequence of star graphs $\Gamma$ has the form depicted in Figure~\ref{fig:big-star}.
\begin{figure}[H]
\begin{tikzpicture}[scale=0.8]
\coordinate (a) at (0,0);
\coordinate (b) at (4,0);
%\draw[fill] (a) circle (1.2pt);
\draw[thick] (a) -- (b);
\draw[thick,fill=white] (a) circle (2pt);
\draw[fill] (b) circle (1.2pt);
\draw[thick] (b) -- (5,0);
\draw[thick] (b) -- (4,1);
\draw[thick] (b) -- (4,-1);
\draw[thick] (b) -- (4.71,0.71);
\draw[thick] (b) -- (4.71,-0.71);
\draw[thick] (b) -- (3.29,0.71);
\draw[thick] (b) -- (3.29,-0.71);
\draw[fill] (5,0) circle (1.2pt);
\draw[fill] (4,1) circle (1.2pt);
\draw[fill] (4,-1) circle (1.2pt);
\draw[fill] (4.71,0.71) circle (1.2pt);
\draw[fill] (4.71,-0.71) circle (1.2pt);
\draw[fill] (3.29,0.71) circle (1.2pt);
\draw[fill] (3.29,-0.71) circle (1.2pt);
\end{tikzpicture}
\caption{A star $\Gamma$ with a single Dirichlet condition at the end of its long edge. (The white circle indicates a Dirichlet condition, black circles indicate Kirchhoff conditions.)}
\label{fig:big-star}
\end{figure}
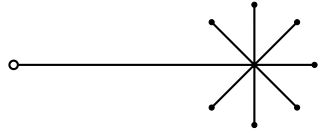

A simple test function argument shows that by adding more short edges to the star $\Gamma$, we can force $\lambda_1 (\Gamma)$ down to zero, while $\lambda_2 (\Gamma)$ remains bounded away from zero (the zero of its eigenfunction must remain on the long edge, meaning $\lambda_2$ cannot drop below the first Dirichlet eigenvalue of an interval of the same length as the long edge). Note that the same argument would work equally well if the short edges were replaced by small loops, creating a graph (a ``bunch of balloons'') with a large number of independent cycles (first Betti number) $\beta$; this was the original example found in \cite{DH10}.

This is still not the full picture, as it raises several natural questions as to when Nicaise' PPW-type bound (Theorem~\ref{thm:nicaise}) does in fact hold, whether Theorem~\ref{thm:nicaise} is actually sharp -- noting that Nicaise' result actually predates the Ashbaugh--Benguria theorem by several years -- , what can be said for more general eigenvalue ratios, and when and how badly the inequality can go wrong on non-Dirichlet non-trees. We note in passing that this latter question can be rather subtle: in \cite{BHY24} the authors find a family of such non-trees, ``saguaro graphs'', for which one still has the bound
\begin{displaymath}
    \frac{\lambda_2(\Gamma)}{\lambda_1(\Gamma)} \leq 2 + \sqrt{5}
\end{displaymath}
(and others, ``ornamented trees'', for which one still has the weaker PPW bound $\frac{\lambda_2(\Gamma)}{\lambda_1(\Gamma)} \leq 5$).

Our goals are thus threefold:
\begin{enumerate}
    \item \label{item:ab-graph} In Theorem 
    \ref{thm:ashbaugh-benguria} we obtain
    an optimal bound of the form
    \begin{equation}
        \label{eq:ab-graph}
        \frac{\lambda_2 (\Gamma)}{\lambda_1 (\Gamma)} \leq 4
    \end{equation}
    for all compact Dirichlet trees $\Gamma$, thus moving from a non-sharp PPW-type bound to a sharp Ashbaugh--Benguria bound, where the constant $4$ is the ratio for any finite interval (a one-dimensional version of the ball). We also characterize the graphs for which there is equality, which turn out to be a far larger category than just intervals;
    \item \label{item:general-indices} We more generally obtain bounds to eigenvalue ratios $\frac{\lambda_k (\Gamma)}{\lambda_j(\Gamma)}$ for compact Dirichlet trees $\Gamma$ and any $k > j$. %In particular, we will establish the existence of a constant $C(k,j)$ independent of $\Gamma$ such that $\frac{\lambda_k (\Gamma)}{\lambda_j(\Gamma)} \leq C(k,j)$, as well as study its dependence on $k$ and $j$; more c
    Concretely, for all compact Dirichlet trees we prove the universal inequality $\frac{\lambda_k(\Gamma)}{\lambda_j(\Gamma)} \leq 4 \cdot \frac{k^2}{j^2}$ for all $k > j \geq 1$ (see Theorem~\ref{thm:dirichlet-tree-2k-k} and Corollary~\ref{cor:dirichlet-tree-k-j});
    \item \label{item:general-graphs} We study the problem of bounding ratios $\frac{\lambda_k (\Gamma)}{\lambda_j(\Gamma)}$, $k > j$, for general compact graphs which may have some number $N \geq 0$ of Neumann leaves and some number $\beta \geq 0$ of independent cycles, proving that, at least for $k > j \geq N + \beta + 1$, the ratio can be bounded from above by a constant $C(k,j,N,\beta)$ independent of $\Gamma$ (in particular, the \emph{only} way that things can ``go wrong'' is via Neumann leaves and/or cycles), and studying in what classes of graphs maximizers and minimizers of such eigenvalue ratios can exist (see Theorems~\ref{thm:general-graph-k-j} and~\ref{thm:existence}, respectively).
\end{enumerate}

\subsection{A bound of Ashbaugh--Benguria-type}

We will first explore \autoref{item:ab-graph} of the list of goals. Since an interval is a one-dimensional version of a ball, the natural version of the Ashbaugh--Benguria bound (or PPW conjecture) for compact Dirichlet trees is as follows. As noted, the class of graphs for which there is equality is, however, much larger.

\begin{theorem}
\label{thm:ashbaugh-benguria}
    Let $\Gamma$ be a compact Dirichlet tree. Then
    \begin{equation}
    \label{eq:ab-conj-tree}
        \frac{\lambda_2 (\Gamma)}{\lambda_1 (\Gamma)} \leq 4,
    \end{equation}
    with equality if and only if $\Gamma$ is an interval or an equilateral star graph (on any number of edges)
\end{theorem}
Vertices of degree two are understood to be suppressed in all theorems of this article, although such ``dummy vertices'' will be introduced for convenience in certain arguments.  We note nonetheless that an equilateral star graph on two edges would reduce to an interval with a dummy vertex, and could be included in this instance.

We will prove this theorem in Section~\ref{sec:ab-dirichlet-trees} via a completely different method from the symmetrization approach taken by Ashbaugh and Benguria. In fact, Theorem~\ref{thm:ashbaugh-benguria} will be quite a direct consequence of the following inequality of some independent interest, which bounds the effect of changing a Dirichlet condition to a Neumann one at a single vertex.

In what follows, for a tree $\Gamma$, we will denote by $\tau_1(\Gamma)$ the first eigenvalue of the Laplacian on $\Gamma$ with Neumann conditions at a single leaf, Dirichlet conditions at all remaining leaves, and Kirchhoff conditions at all other vertices. That is, in passing from $\lambda_1$ to $\tau_1$ we have changed, at an arbitrary leaf not specified in the notation, a single Dirichlet condition to a Neumann condition.

\begin{theorem}
\label{thm:vc}
Let $\Gamma$ be a compact tree. Then, with the above notation,
\begin{equation}
\label{eq:vc-change}
    \frac{\lambda_1(\Gamma)}{\tau_1(\Gamma)} \leq 4,
\end{equation}
with equality if and only if $\Gamma$ is an interval.
\end{theorem}

The proof of Theorem~\ref{thm:vc}, also given in Section~\ref{sec:ab-dirichlet-trees} (although some of the technical parts of the proof are in Appendix~\ref{sec:appendix}), adapts a kind of optimization argument similar to one sometimes used for shape optimization of eigenvalues on domains and manifolds. We will consider the effect of edge length perturbations of $\Gamma$ on the eigenvalue ratio in \eqref{eq:vc-change}, more precisely showing that, for any fixed graph topology (or ``underlying discrete graph'') other than an interval, no critical point of the functional $\frac{\lambda_1(\cdot)}{\tau_1(\cdot)}$ can be a maximum if all edge lengths are positive. On the other hand, as will be fairly easy to show via a compactness argument, this ratio must attain a maximum and a minimum among all quantum graphs with the same underlying discrete graph, if edges of length zero are allowed. Hence the only possibility for the maximum will be a degeneracy, which will allow us to conclude the bound via an induction over the number of edges of the graph.

To the best of our knowledge this approach is new in the context of quantum graphs, although the principle of maximizing or minimizing eigenvalues of quantum graphs under edge length perturbations for a fixed graph topology is well established, having been introduced and thoroughly studied in \cite{BL17} in the fundamental case of the first nontrivial eigenvalue of the Laplacian with Kirchhoff conditions.

\subsection{General eigenvalue ratios for Dirichlet trees}

Regarding \autoref{item:general-indices} from our list of goals, our main result is as follows.

\begin{theorem}
\label{thm:dirichlet-tree-2k-k}
    Let $\Gamma$ be any compact Dirichlet tree. Then, for any $k \geq 1$,
    \begin{equation}
    \label{eq:dirichlet-tree-2k-k}
        \frac{\lambda_{2k} (\Gamma)}{\lambda_k (\Gamma)} \leq 4.
    \end{equation}
\end{theorem}

The inequality \eqref{eq:dirichlet-tree-2k-k} is sharp; for any given $k \geq 1$ there is equality if $\Gamma$ is an equilateral star on $2k$ edges (or if $\Gamma$ is an interval), cf.\ also \cite[Section 15, open problem (iii)]{AB07}.  The inequality can also be interpolated to cover any pair of eigenvalues, and also implies a version, for Dirichlet trees, of many results for eigenvalue ratios, the analogues of which remain conjectures on domains.

\begin{corollary}
\label{cor:dirichlet-tree-k-j}
    Let $\Gamma$ be any compact Dirichlet tree. Then for any $k > j \geq 1$ we have
    \begin{equation}
        \label{eq:ab-kj}
        \frac{\lambda_k (\Gamma)}{\lambda_j (\Gamma)} 
        \leq 4^{\lceil\log_2 \frac{k}{j}\rceil} \leq 4\cdot\frac{k^2}{j^2}.
    \end{equation}
\end{corollary}

Theorem~\ref{thm:dirichlet-tree-2k-k} and Corollary~\ref{cor:dirichlet-tree-k-j} will be proved in Section~\ref{sec:general-dirichlet-trees}.

\begin{remark}
\label{rem:big-ab}
\begin{enumerate}
    \item We note that by iterating Nicaise' result, Theorem~\ref{thm:nicaise}, we can already deduce, for any $k > j$, the existence of an absolute constant $C(k,j) > 0$ such that
    \begin{equation}
    \label{eq:nicaise-kj}
        \frac{\lambda_k (\Gamma)}{\lambda_j (\Gamma)} \leq C(k,j);
    \end{equation}
    namely, we could take $C(k,j) := 5^{k-j}$. However, this bound grows exponentially in $k-j$, whereas in \eqref{eq:ab-kj}, up to a constant, the bound respects the Weyl asymptotics $\lambda_k (\Gamma) \sim \frac{\pi^2 k^2}{L^2}$, where $L = |\Gamma|$ is the total length of $\Gamma$.

    \item On the other hand, it is not clear whether the constant $4$ in \eqref{eq:ab-kj} is optimal, and we leave this as an open problem. Theorem~\ref{thm:ashbaugh-benguria} states that for the specific pair $k=2$, $j=1$ one can replace $4$ by $1$. On the other hand, if $\Gamma$ is an equilateral star on $m$ edges, then $\lambda_m (\Gamma) = \frac{\pi^2m^2}{|\Gamma|^2}$, $\lambda_{m+1} (\Gamma) = \frac{9\pi^2 m^2}{4|\Gamma|^2}$. 
    (Indeed, if the graph is normalized such that each edge has length $\frac{|\Gamma|}{m} = \pi$, then the first eigenfunction has the form $\psi_1(x) = \cos \frac{x}{2}$ on each edge $e \sim [0,\pi]$, the second up to the $m$-th take the form $\psi(x) = \pm \sin x$, or $0$, on each edge, and the ($m+1$)-st eigenfunction, unique up to scalar multiples, will be $\psi_{m+1}(x) = \cos(\frac{3x}{2})$ on each edge.) It follows that
    \begin{displaymath}
        \frac{\lambda_{m+1}(\Gamma)}{\lambda_m (\Gamma)} = \frac{9}{4} = \left(\frac{9m^2}{4(m+1)^2}\right) \cdot \frac{(m+1)^2}{m^2}.
    \end{displaymath}
    Since $m$ may be taken as large as we please, the best possible factor $C$ in the inequality $\frac{\lambda_k(\Gamma)}{\lambda_j (\Gamma)} \leq C \cdot \frac{k^2}{j^2}$ independent of $k$, $j$ and $\Gamma$ must be at least $\frac{9}{4}$. In particular, the ``worst-case scenario'' (the maximal value) is larger than the asymptotic value for any fixed graph coming from the Weyl asymptotics, that is, that $\frac{j^2\lambda_k(\Gamma)}{k^2\lambda_j(\Gamma)} \to 1$ as $k,j \to \infty$, for $\Gamma$ fixed.
    
    A heuristic argument suggests that $4$ may not be optimal: On an $m$-equilateral star, the ratio $\frac{j^2\lambda_k}{k^2\lambda_j}$ appears to attain its maximum precisely when $k=m+1$, $j=m$, and we also note that these equilateral star graphs are extremal in several relevant ways (including the case $k=2$, $j=1$, but also as maximizers of the fundamental spectral gap of the Laplacian with all standard vertex conditions among all graphs of fixed length and fixed number of edges \cite[Theorem~4.2]{KKMM16}). So it seems quite plausible that these graphs could also be maximizers (in the limit as $m \to \infty$) in this case as well, meaning that the correct factor in \eqref{eq:ab-kj} would be $\frac{9}{4}$.
    
    \item
    Finally, we note that Theorem~\ref{thm:dirichlet-tree-2k-k} immediately implies versions for Dirichlet trees of several conjectures for eigenvalue ratios on domains (see \cite[Section~8.6]{A17} or \cite[Section~4]{A00}), in particular, for any $m \geq 1$,
    \begin{equation}
        \label{eq:ab-b1}
        \frac{\lambda_{m+1}(\Gamma)}{\lambda_m(\Gamma)} \leq \max_\Gamma \frac{\lambda_2(\Gamma)}{\lambda_1(\Gamma)} = 4,
    \end{equation}
    the value for an interval. This is also better than the previous best known result due to Nicaise; see Theorem~\ref{thm:nicaise}.
    It is also the best possible absolute estimate independent of $m$, since it is sharp for the interval and $m=1$. 
\end{enumerate}
\end{remark}

%At any rate, even if the constant in \eqref{eq:ab-kj} turns out not to be sharp, we thus see that, at least on Dirichlet trees, the eigenvalues obey a very strong set of bounds. This is to be compared with what is known on domains, where $\lambda_4/\lambda_2$ satisfies the same bound as $\lambda_2/\lambda_1$ (another result of Ashbaugh and Benguria, see \cite[Section~8.6]{A17}), but similar bounds for ratios of arbitrary higher eigenvalues are, to the best of our knowledge, unknown (although, again, see \cite[Section~8.6]{A17} for an extensive discussion of bounds of the form $\lambda_{k+1}/\lambda_k$ and $\lambda_k/\lambda_1$ on domains, with references).

At any rate, even if the constant in \eqref{eq:ab-kj} turns out not to be sharp, we thus see that, at least on Dirichlet trees, the eigenvalues obey a very strong set of bounds, much stronger than what is known on domains. Nevertheless, the key bound \eqref{eq:dirichlet-tree-2k-k} does have \emph{an} analogue of sorts on domains, namely for the ratio of eigenvalues whose eigenfunctions have, respectively, $2k$ and $k$ nodal domains, which in general will not be $\lambda_{2k}$ and $\lambda_k$. Our proof will exploit the property that, after an arbitrarily small perturbation of the edge lengths of $\Gamma$ (a perturbation under which the eigenvalues are stable), for all $j \geq 1$ the $j$-th eigenfunction of $\Gamma$ has exactly $j$ nodal domains (cf.\ \cite{B07,B08}), together with the Ashbaugh--Benguria bound of Theorem~\ref{thm:ashbaugh-benguria}.

\begin{remark}
For some purposes Dirichlet trees are natural analogues of Dirichlet Laplacians on domains. In this context we observe that \cite[Theorem~4.7]{BKKM17} contains the following bound as a special case: for any compact Dirichlet tree $\Gamma$ of total length $L > 0$ and any $k \geq 1$,
    \begin{equation}
    \label{eq:dirichlet-tree-polya}
        \lambda_k (\Gamma) \geq \frac{\pi^2 k^2}{L^2}.
    \end{equation}
This is a version of P\'olya's conjecture for Dirichlet trees: for any Dirichlet tree $\Gamma$ and any $k \geq 1$, the principal term in the Weyl asymptotics, $\frac{\pi^2 k^2}{L^2}$, \emph{always} gives a lower bound on $\lambda_k(\Gamma)$.  Compare with Friedlander's bound of $\frac{\pi^2 k^2}{4L^2}$ for the $(k+1)$-st eigenvalue of the Laplacian with Kirchhoff conditions \cite{F05}, but equally valid for the $k$-th eigenvalue in the presence of at least one Dirichlet condition.

Actually, \cite[Theorem~4.7]{BKKM17} is a more general result, and the special inequality \eqref{eq:dirichlet-tree-polya} went unnoticed there. We also note that the case of equality has not been investigated, something that we will not do here; but we expect that there is equality (for a given pair of a graph $\Gamma$ and an index $k$) if and only if every edge of $\Gamma$ has length equal to a multiple of $\frac{L}{k}$. (Certainly, the ``if'' direction is immediate, and the ``only if'' direction is not hard to show for $k=1,2$.)
\end{remark}

Eq. \eqref{eq:dirichlet-tree-polya} is significant for an additional reason beyond its relation to the Weyl asymptotics: The sharper bound for mixed Dirichlet and Kirchhoff conditions found in \cite[Theorem~4.7]{BKKM17}, of which \eqref{eq:dirichlet-tree-polya} is a special case, reads
\begin{equation}
\label{eq:bkkm-lower}
    \lambda_k (\Gamma) \geq \left(k - \frac{N+\beta}{2}\right)^2 \frac{\pi^2}{L^2}
\end{equation}
for $k \geq 2$ (or $k \geq 1$ if there is at least one Dirichlet vertex), where, as before, $N$ is the number of Neumann leaves and $\beta$ the number of independent cycles (for \eqref{eq:dirichlet-tree-polya} just take $N=\beta=0$). As with the case of eigenvalue ratios, it is notable that \emph{only} these two quantities are responsible for ``deviations'' from the ``baseline'' lower bound coming from the Weyl asymptotics. It would be interesting to explore whether there are other senses in which Dirichlet trees are, as here, a natural analogue of Dirichlet Laplacians on Euclidean domains.

\subsection{Bounds and existence of maximizers for general graphs}

We now turn to \autoref{item:general-graphs}, and assume that $\Gamma$ may have cycles and may have any mix of Dirichlet and Kirchhoff conditions at its leaves, if it has any. As sketched above, it is known that no bound of the form \eqref{eq:ab-graph}, and thus no bound of the form \eqref{eq:nicaise-kj}, is possible for a constant $C(k,j)$ independent of the graph. However, we can use an upper bound from \cite[Theorem~4.9]{BKKM17} originally due to Ariturk \cite{A16}, namely
\begin{displaymath}
    \lambda_k (\Gamma) \leq \frac{\pi^2 (k-2+\beta+D+\frac{N+\beta}{2})^2}{L^2}
\end{displaymath}
for any $k\geq 1$, where as before $N$ is the number of Neumann leaves, $\beta$ is the number of independent cycles, and $D$ is the number of Dirichlet leaves, together with  the Faber--Krahn-type lower bound of Nicaise \cite[Th\'eor\`eme~3.1]{N87}, which shows that any nonzero eigenvalue is at least $\frac{\pi^2}{4L^2}$, to obtain that
\begin{equation}
    \label{eq:general-graph-dnb}
    \frac{\lambda_k (\Gamma)}{\lambda_j (\Gamma)} \leq 4\left(k-2+\beta+D+\frac{N+\beta}{2}\right)^2 =: C(k,j,N,D,\beta)
\end{equation}
for any $k > j \geq 1$. Of course this bound is likely to be extremely poor in general; among other things, by estimating $\lambda_j$ from below by $\lambda_1$ we are suppressing the dependence on $j$. But the point is that, for any pair $(k,j)$, there \emph{exists} an upper bound depending (at most) on this pair of indices, as well as $N$, $D$ and $\beta$. However, as a consequence of Corollary~\ref{cor:dirichlet-tree-k-j} (or, indeed, \eqref{eq:nicaise-kj}) and a simple \emph{surgery principle} (namely interlacing inequalities for eigenvalues under changing vertex conditions, see \cite[Theorem~3.4]{BKKM19}), we can easily suppress the dependence of the constant on $D$ (and generally improve the dependence on $j$), to obtain the following:

\begin{theorem}
    \label{thm:general-graph-k-j}
    Let $\Gamma$ be a compact graph equipped with any combination of Neumann, Kirchhoff, and Dirichlet conditions, and let $k > j \geq N + \beta + 1$. %Then there exists a constant $C=C(k,j,N,\beta)$ depending only on the number of Neumann leaves, $N$, and the number of independent cycles, $\beta$, in addition to the indices $k$ and $j$, such that
    %\begin{displaymath}
    %    \frac{\lambda_k (\Gamma)}{\lambda_j (\Gamma)} \leq C(k,j,N,\beta).
    %\end{displaymath}
    Then
    \begin{displaymath}
        \frac{\lambda_k (\Gamma)}{\lambda_j (\Gamma)} \leq 4\cdot \frac{k^2}{(j-(N+\beta))^2}
    \end{displaymath}
\end{theorem}

Together, this result and Corollary~\ref{cor:dirichlet-tree-k-j} give an essentially complete general answer to what is driving the existence of quantum graph counterexamples to bounds of Ashbaugh--Benguria (or PPW) type for ratios of eigenvalues. That said, Theorem~\ref{thm:general-graph-k-j} should also be true (with an adjusted upper bound $C(k,j,N,\beta)$) for all $j \leq N + \beta$ for which $\lambda_j(\Gamma) \neq 0$, although this would require a completely different method of proof, and we will not explore the question further here.

%\JK{Check: since given $N$ and $\beta$ this only rules out finitely many $j$, maybe there is still a way around the former condition. Also: the assumption that $j \geq N + \beta + 1$ automatically implies that $\lambda_j (\Gamma) > 0$; if $j \leq N + \beta$ this would need to be an additional assumption, since it is possible that $\lambda_1 (\Gamma) = 0$.}
%\JK{If the conjecture is true, we should have $\frac{\lambda_k (\Gamma)}{\lambda_j (\Gamma)} \leq 4\cdot \frac{k^2}{(j-(N+\beta))^2}$, where for the lower bound on $\lambda_j(\Gamma)$ we turn the $N$ Neumann leaves into Dirichlet ones, and break each of the $\beta$ cycles by placing a Dirichlet condition somewhere on the cycle; then use the interlacing inequality that we have shifted down by at most $N+\beta$ positions. Without the conjecture the best we can do would be $(C_{AB})^{\lceil \log_2(\frac{k}{j-(N+\beta})\rceil}$.}
%In principle the constant is explicit (see \eqref{eq:general-graph-explicit}), though as it depends on $C_{AB}$ it will only have an essentially optimal form if Conjecture~\ref{conj:ashbaugh-benguria} is true. But we emphasize that even without that conjecture, this result and Corollary~\ref{cor:dirichlet-tree-k-j} together give an essentially complete general answer to what is driving the existence of quantum graph counterexamples to PPW-type bounds for ratios of eigenvalues.

In a related spirit, we examine under what circumstances can we expect an eigenvalue ratio to attain a minimum and a maximum, within a given class of quantum graphs. We first note that the minimization problem is trivial, even among Dirichlet trees, so we will otherwise restrict ourselves to maximization:

\begin{proposition}
\label{prop:ratio-inf}
    Fix $k > j \geq 1$. Then
    \begin{displaymath}
        \inf \left\{ \frac{\lambda_k (\Gamma)}{\lambda_j (\Gamma)} : \Gamma \text{ compact Dirichlet tree} \right\} = 1.
    \end{displaymath}
    This infimum is always attained: for $j \geq 2$ on an equilateral $k$-star; for $j=1$ only on graphs which have a Dirichlet condition at an interior (non-leaf) vertex, and are thus effectively disconnected.  In particular the infimum is attained on any equilateral $k$-star with a Dirichlet condition at its center vertex.
\end{proposition}

Obviously, the same result trivially holds if we enlarge our class of admissible graphs to include graphs with cycles and/or Neumann leaves.

Otherwise, if we bound the complexity of the graph, in the form of the number of vertices, from above, then we can use the continuity of the eigenvalues with respect to edge length perturbations (as shown in \cite{BLS19}, see Section~\ref{sec:edge-length-perturbations} for more details) to obtain existence of minimizers and maximizers.

\begin{theorem}
\label{thm:existence}
Fix $m \geq 2$ and consider the set of compact, connected quantum graphs that have at most $m$ edges, with any combination of Dirichlet and Neumann conditions at any leaves. Then for any pair $1 \leq j < k$ (where if $j=1$, we also assume that there is at least one Dirichlet leaf), there exist a graph minimizing the ratio $\frac{\lambda_k}{\lambda_j}$ and another graph maximizing this ratio, although the minimizing/maximizing graph may have a Dirichlet condition at a non-leaf.
\end{theorem}

The assumption that there is at least one Dirichlet leaf if $j=1$ is made to control $\lambda_1$ from below. At any rate, uniqueness is not guaranteed.  (Indeed, as long as $m > k$, for any $1 < j < k$ any equilateral $n$-star, $k \leq n \leq m-1$, will serve as a minimizer for the ratio $\frac{\lambda_k}{\lambda_j}$, cf.\ Proposition~\ref{prop:ratio-inf}.) Note that it is not sufficient just to bound the number of \emph{vertices} from above, as opposed to the number of edges, as the ``bunch of balloons'' example from \cite[Remark (2) after Example 1.2]{DH10} shows.

\begin{corollary}
\label{cor:existence}
Fix $D \geq 0$, $N \geq 0$ and $\beta \geq 0$ and consider the set of compact, connected quantum graphs that have at most $D$ Dirichlet leaves, $N$ Neumann leaves, and $\beta$ independent cycles. Then the conclusion of Theorem~\ref{thm:existence} holds verbatim.
\end{corollary}

In particular, we may, if we wish, restrict to considering only graphs having Dirichlet conditions at the leaves, or (if $j \geq 2$) only Neumann conditions at the leaves. The major question is whether the restriction on the number of edges in Theorem~\ref{thm:existence}, or alternatively the three restrictions in Corollary~\ref{cor:existence}, can be weakened to a restriction on just $N$ and $\beta$, independent of the number of Dirichlet leaves. However, proving this would likely require a completely new approach, so we do not attempt it here.

We will prove Theorem~\ref{thm:general-graph-k-j}, Proposition~\ref{prop:ratio-inf}, Theorem~\ref{thm:existence} and Corollary~\ref{cor:existence} in Section~\ref{sec:general-graphs}.

We finish the introduction by collecting, and making more explicit, the open problems mentioned throughout.

\begin{problem}
\begin{enumerate}
    \item Study the problem of maximizing the ratios $\frac{\lambda_k}{\lambda_1}$ on Dirichlet trees. In particular, as a $1$-dimensional analogue of \cite[Open problem 14]{H06}, we can expect that $\frac{\lambda_3}{\lambda_1}$ should be maximized by the interval. (See Remark~\ref{rem:big-ab}(3).) Numerical evidence suggests that even for general $k \geq 4$, $\frac{\lambda_k}{\lambda_1}$ might attain its maximum on the interval.
    \item Determine the optimal constant $\frac{9}{4} \leq C \leq 4$ in the bound $\frac{\lambda_k(\Gamma)}{\lambda_j (\Gamma)} \leq C \cdot \frac{k^2}{j^2}$ for all Dirichlet trees $\Gamma$ and all $k,j \geq 1$. Is it $\frac{9}{4}$, corresponding to the limit for equilateral stars as the number of edges diverges to $\infty$? (See Remark~\ref{rem:big-ab}(2).)
    \item Characterize the case of equality in the P\'olya-type bound \eqref{eq:dirichlet-tree-polya}, $\lambda_k(\Gamma) \geq \frac{\pi^2 k^2}{L^2}$, for all $k \geq 1$ and all Dirichlet trees $\Gamma$ of total length $L$. Is it true that there is equality if and only if all edges of $\Gamma$ have length equal to an integer multiple of $\frac{L}{k}$?
    \item Obtain an upper bound on $\frac{\lambda_k(\Gamma)}{\lambda_j(\Gamma)}$, for any compact quantum graph $\Gamma$ and any pair $k \geq j$, where the upper bound should only depend on the number of Neumann leaves $N$ and the number of independent cycles $\beta$ (cf.\ Theorem~\ref{thm:general-graph-k-j}).
    \item Study whether, for any given $N \geq 0$ and $\beta \geq 0$, among all compact graphs with at most $N$ Neumann leaves and $\beta$ independent cycles, for any $k \geq j \geq 1$ there exists a graph maximizing $\frac{\lambda_k}{\lambda_j}$ (independent of the number of Dirchlet leaves) cf.\ Corollary~\ref{cor:existence}.
\end{enumerate}
\end{problem}

\section{Notation and preliminaries}
\label{sec:prelim}

\subsection{Metric graphs}

Throughout, we will consider only compact metric graphs, that is, graphs $\Gamma = (\mathcal{V},\mathcal{E})$ on a finite edge set $\mathcal{E}$, of cardinality $m:=|\mathcal{E}|$ (and a finite vertex set $\mathcal{V}$ of cardinality $n:=|\mathcal{V}|$) where every edge $e \in \mathcal{E}$ can be identified with a finite interval $e \simeq [0,\ell_e]$, and is incident with two vertices $v_1,v_2 \in \mathcal{V}$, corresponding to the endpoints of the interval. Note that, in general, we permit loops (edges $e$ for which $v_1=v_2$ in the above notation) and parallel edges. Such a graph is a (compact) metric measure space when equipped with the Lebesgue measure induced naturally by the Lebesgue measure on each interval and the natural Euclidean metric induced by the intervals and the identification of their endpoints to form vertices. This metric measure space, up to isometric isomorphism, is independent of the choice of interval used to represent each edge. See \cite[Section~1.3]{BK13} or \cite{M19} for more details.

We start with some basic graph notation and nomenclature. We will assume without further comment that all our graphs are connected unless explicitly stated otherwise.

\begin{definition}
Let $\Gamma = (\mathcal{V},\mathcal{E})$ be a compact metric graph.
\begin{enumerate}
\item The total length of the graph will be denoted by $L:=|\Gamma|=\sum_{e \in \mathcal{E}} \ell_e$.
\item The (first) \emph{Betti number} of the graph, or the number of independent cycles in $\Gamma$, will be denoted by $\beta:=m-n+1 = |\mathcal{E}| - |\mathcal{V}| + 1$.
\item We call $\Gamma$ a \emph{tree} if $\beta = 0$.
\item The \emph{degree} of a vertex $v \in \mathcal{V}$, $\deg v$, is the number of edges incident with $v$ (where a loop counts twice).
\item We call a vertex $v \in \mathcal{V}$ a \emph{leaf} if $\deg v = 1$. Note that this term will be applied to general graphs, not just to trees.
\item We call a vertex $v \in \mathcal{V}$ a \emph{dummy vertex} if $\deg v = 2$.
\end{enumerate}
\end{definition}

\subsection{Function spaces and Laplacians}

We   use standard notation for spaces of functions on $\Gamma$: for $1 \leq p \leq \infty$, $L^p(\Gamma) \simeq \bigoplus_{e \in \mathcal{E}} L^p(0,\ell_e)$ is the space of $p$-integrable functions (or essentially bounded functions if $p=\infty$), $C(\Gamma)$ is the space of continuous functions on $\Gamma$ (that is, edgewise continuous up to the vertices, and with a well-defined value at each vertex), and
\begin{displaymath}
    H^1(\Gamma) = \{ f \in C(\Gamma): f|_e \in H^1(e) \text{ for all } e \in \mathcal{E} \}
\end{displaymath}
is the space of square-integrable functions with square-integrable weak derivative on $\Gamma$. We will also use the notation $H^1_0(\Gamma,V_D)$ for the subspace of functions in $H^1(\Gamma)$ which vanish on some finite set $V_D \subset \Gamma$ of points in $\Gamma$ (which without loss of generality we can assume to be vertices), or $H^1_0(\Gamma)$ if there is no ambiguity as to which set $V_D$ is meant (most commonly, but not necessarily, this will be the set of leaves of $\Gamma$).

For such a set $V_D$, we define the Laplacian on $\Gamma$ as the operator $-\Delta$ on $L^2(\Gamma)$ associated with the sesquilinear form
\begin{displaymath}
    h [f,g] = \int_\Gamma f' \bar{g}'\,\text{d}x
\end{displaymath}
with form domain $H^1_0(\Gamma,V_D)$. As is well known (and can be found in many places, including \cite[Chapter~1]{BK13}), this operator is self-adjoint and bounded from below, and has compact resolvent, and thus a discrete spectrum consisting only of eigenvalues, whose finite algebraic and geometric multiplicities coincide, and which will be denoted by
\begin{displaymath}
    \lambda_1(\Gamma) < \lambda_2(\Gamma) \leq \lambda_3(\Gamma) \leq \ldots \to \infty.
\end{displaymath}
We will generally denote the (or rather an) eigenfunction associated with $\lambda_k(\Gamma)$ by $\psi_k$, and will generally choose these eigenfunctions to form an orthonormal basis of $L^2(\Gamma)$.

It is also well known that at all vertices outside $V$, functions $f$ in the domain of $-\Delta$ satisfy, in addition to continuity, the Kirchhoff condition
\begin{displaymath}
    \sum_{e\sim v} \partial_\nu f|_e(v) = 0,
\end{displaymath}
where $\partial_\nu f|_e(v)$ is the derivative of $f|_e \in H^2(e)$ at the endpoint of $e$ corresponding to the vertex $v$, oriented to point into $v$. At the vertices in $V$ the operator satisfies a zero, or Dirichlet, condition.

If $\Gamma$ is a tree and $V$ is its set of leaves, then we will speak of the Dirichlet Laplacian and ``Dirichlet leaves''; if $v \not\in V_D$ has degree $1$, then the Kirchhoff condition reduces to the Neumann condition $\partial_\nu f|_e (v) = 0$, in which case we will speak of ``Neumann leaves''. If $V_D = \emptyset$ and we have the Laplacian with only standard Kirchhoff conditions, then $\lambda_1(\Gamma) = 0$; otherwise, if $V_D \neq\emptyset$, necessarily $\lambda_1(\Gamma)>0$, as is well known.

%If $V_D = \emptyset$, we will speak of the Laplacian with standard or Neumann--Kirchhoff conditions, for short the standard Laplacian, whose eigenvalues we will also denote by
%\begin{displaymath}
%    0 = \mu_1(\Gamma) < \mu_2(\Gamma) \leq \mu_3(\Gamma) \leq \ldots \to \infty,
%\end{displaymath}
%the eigenfunctions associated with $\mu_1(\Gamma)$ being the constant functions.

We will occasionally need to impose delta, or Robin, conditions at one or more points (again, without loss of generality vertices). For such a finite set $V_R$, for each $v \in V_R$ we choose a number $\alpha_v \in \R$, write $\alpha$ for the vector of $\alpha_v$ (or, in an abuse of notation, $\alpha \in \R$ if all $\alpha_v$ coincide) and define the form
\begin{displaymath}
    h_\alpha [f,g] = \int_\Gamma f'\bar g'\,\text{d}x + \sum_{v \in V_R} \alpha_v f(v)\overline{g(v)}
\end{displaymath}
for all $f,g \in H^1_0(\Gamma,V_D)$. The associated Laplacian on $L^2(\Gamma)$ enjoys essentially the same properties described above; we will tend to denote its eigenvalues by $\lambda_k^\Gamma(\alpha)$, as in such cases we will be primarily interested in the dependence of the eigenvalues on the parameter(s) $\alpha$. (We stress that $V$ may be arbitrary as long as it is finite and $V_D \cap V_R = \emptyset$, that is, any mix of Dirichlet and standard conditions is allowed at the other vertices.) At any vertex $v \in V_R$ the corresponding Robin-type condition, or delta condition, reads
\begin{displaymath}
    \sum_{e\sim v} \partial_\nu f|_e(v) + \alpha_v f(v) = 0
\end{displaymath}
under our choice of convention on $\nu$. We will also speak of a delta potential (of strength $\alpha_v$) at $v$. If $\alpha_v = 0$, then we recover standard conditions at the vertex, while $\alpha_v=\infty$ corresponds formally to a Dirichlet condition.

In all the above cases, the eigenvalues admit the usual variational characterization via the associated Rayleigh quotient
\begin{displaymath}
    R[f] = \frac{h[f,f]}{\|f\|_{L^2(\Gamma)}^2}
\end{displaymath}
for $f \in H^1_0(\Gamma,V_D)$ (or $R_\alpha[f] = \frac{h_\alpha[f,f]}{\|f\|_{L^2(\Gamma)}^2}$ in the presence of one or more Robin vertices).

We do not go into further details, which may be found in multiple sources, including \cite{BK13,Ku24} and also \cite[Section~2]{BKKM19}.

Finally, it is well known that the insertion (or deletion) of a dummy vertex equipped with Kirchhoff conditions induces an isometric isomorphism at the level of all the function spaces we are considering, and a unitary equivalence at the level of the operators; moreover, all the geometric quantities of interest (the number of Dirichlet and Neumann leaves, the number of independent cycles, the total length) are unaffected (cf.\ \cite[Assumption~3.1 and the discussion after it]{BKKM17} or \cite[Remark~2.1]{BKKM19}, for example). Thus, as usual, we will treat any given point of a graph as a vertex if it is convenient to do so, and otherwise suppress all dummy vertices.

\subsection{Edge length perturbations}
\label{sec:edge-length-perturbations}

Given a metric graph $\Gamma$ on $m:=|E|$ edges, we may associate an underlying discrete graph $G = (V,E)$, where now every edge $e \in E$ corresponds merely to a relation $e \simeq (v_1,v_2)$ between two vertices $v_1,v_2 \in V$. Indeed, $\Gamma$ may also be considered as a pair $(G,{\bf l})$, where $G$ is the underlying discrete graph and ${\bf l} = (\ell_{e_1},\ldots,\ell_{e_m}) \in \R^m_+ \setminus \{0\}$ is a vector of edge lengths, the numbering of the edges being fixed in function of $G$ for the purposes of the identification. For practical purposes, it will be important to allow $\ell_{e_i} = 0$ for one or more edges, in which case the two adjacent vertices will coincide in $\Gamma$, that is, $\R^m_+$ is taken as the \emph{closed} positive quadrant in $m$-dimensional space, from which we remove only the origin (where all edge lengths would be zero).

\begin{definition}
\label{def:underlying}
Let $\Gamma$ be a compact metric graph, and let $G$ be a discrete graph.
\begin{enumerate}
\item We say that $G$ is a \emph{proper} underlying discrete graph if, with the above notation, $\ell_{e} > 0$ for all $e \in E$. Note that $\Gamma$ may have multiple underlying discrete graphs, but up to a renumbering of the edges only one proper underlying discrete graph.
\item Let $G$ be a proper underlying discrete graph for $\Gamma$. We denote by $\GammaClass$ the set of all metric graphs which have $G$ as a (not necessarily proper) underlying discrete graph.
\item We denote by $\GClass$ the set of all compact metric graphs which have $G$ as an underlying discrete graph, proper or otherwise.
\end{enumerate}
\end{definition}

Thus if $\Gamma' \in \GammaClass$, then $\Gamma'$ can be obtained from $\Gamma$ simply by varying the edge lengths in $\Gamma$, whereby some edge lengths may shrink to zero when passing from $\Gamma$ to $\Gamma'$.

However, for any two compact metric graphs $\Gamma,\Gamma'$ one may find a discrete graph $G$ such that $\Gamma,\Gamma' \in \GClass$ (just glue $\Gamma$ and $\Gamma'$ at a single vertex, and take $G$ to be the proper underlying discrete graph of the glued graph).

One of the main reasons for considering such classes of graphs is that, as is now well established, the eigenvalues are continuous (indeed, mostly differentiable) functions of the edge lengths. For this we need a slight modification of the classes, to allow for the differential operator, which in our case is determined by the vertex conditions we impose.

\begin{definition}
\label{def:underlying-vc}
Given a discrete graph $G$ with $n$ vertices $v_1,\ldots,v_n$ and a vector $\alpha \in (\R\cup \infty)^n$, we assume that, the vertex $v_i$ is equipped with a potential of strength $\alpha_i$, $i=1,\ldots,n$ (under the convention that $\alpha_i=0$ returns a standard condition and $\alpha_i=\infty$ is used for a Dirichlet condition).

We will denote by $\mathcal G_{G,\alpha}$ the set of all quantum graphs $\Gamma$ with underlying discrete graph $G$, equipped with the Laplacian with vertex conditions specified by the vector $\alpha$, where if one or more edges of $\Gamma \in \mathcal G_{G,\alpha}$ adjacent to $v$ has length zero, then the correct vertex condition to impose at $v$ is equal to the \emph{sum} $\sum_w \alpha_w$ taken over all vertices adjacent to $v$ in $G$ and coincident with $v$ in $\Gamma$ (in particular, a Dirichlet condition is imposed at $v$ in $\Gamma$ if it is imposed at any of the $w$ in the class $\mathcal G_{G,\alpha}$)
\end{definition}

In a slight abuse of notation, we will often not distinguish between $\GClass$, the set of all metric graphs with a given topology, and $\mathcal G_{G,\alpha}$, the set of all quantum graphs with that topology and pre-imposed vertex conditions.

\begin{theorem}
\label{thm:edge-length-continuity}
Let $G$ be a discrete graph with $m$ edges and $n$ vertices, and fix any mix of vertex conditions corresponding to a vector $\alpha \in (\R\cup\infty)^n$.
\begin{enumerate}
\item Identify any $\Gamma = \Gamma(\ell) \in \mathcal G_{G,\alpha}$ with its vector of edge lengths $\ell \in \R^m_+ \setminus \{0\}$. Then for each $k \in \N$, and with the above convention on the vertex conditions, $\ell \mapsto \lambda_k(\ell) := \lambda_k(\Gamma(\ell))$ is a continuous function on $\R^m_+ \setminus \{0\}$.
\item Let $\lambda_k(\Gamma_0)$ be a simple eigenvalue of $\Gamma_0 \in \mathcal G_{G,\alpha}$, where $G$ is assumed proper, i.e., $\ell_e(\Gamma_0) > 0$ in $\Gamma_0$ for all $e \in E$. Denote by $\psi_k$ an associated eigenfunction, chosen to have $L^2$-norm $1$. Then, for any fixed $e \in E$, the map $\ell_e \mapsto \lambda_k(\Gamma)$, $\Gamma \in \mathcal G_{G,\alpha}$, is differentiable at $\Gamma_0$, with
\begin{displaymath}
    \frac{{\rm d}}{{\rm d}\ell}\lambda_k(\Gamma_0) = -\psi_k'(x)^2 - \lambda_k(\Gamma_0)\psi_k(x)^2
\end{displaymath}
for any $x$ in the interior of the edge $e$ in $\Gamma_0$ (this quantity being independent of $x \in e$).
\end{enumerate}
\end{theorem}

The first statement is contained in \cite[Lemma~3.4 and Theorem~3.6]{BLS19}. The second statement is often called a
{\it Hadamard formula}, by way of analogy with the formulas for the shape derivative of Laplacian eigenvalues defined on Euclidean domains. It was first proved for standard conditions only \cite[Proof of the lemma]{F05b}, but is valid in the more general case with essentially the same proof, see \cite[Remark~3.14]{BKKM19}.

\subsection{Existence of minimizers and maximizers}
\label{sec:existence}

Before proceeding, we will use the previous theorem to give a basic existence result for graphs maximizing and minimizing any given eigenvalue ratio among all graphs with the same topology. This is very much in the spirit of \cite{BL17}, although here the scaling (total length) of the graph is irrelevant. In what follows we will use the notation and terminology from Definitions~\ref{def:underlying} and~\ref{def:underlying-vc}.

\begin{proposition}
\label{prop:existence-optimizers}
    Fix any discrete graph $G$ and any pair of indices $k, j \geq 1$, assume that only Dirichlet and standard vertex conditions are imposed, and, if standard Kirchhoff conditions apply at all vertices, assume further that $j \geq 2$.
    \begin{enumerate}
    \item Assume that for, each vertex $v$ of $G$, a fixed vertex condition of either Dirichlet or standard type has been specified, corresponding to a vector $\alpha$ with $\alpha_i \in \{0,\infty\}$ for all $i$, and suppose that $k > j$. Then there exist $\Gamma^\ast, \Gamma_\ast \in \mathcal G_{G,\alpha}$ such that
    \begin{displaymath}
        \frac{\lambda_k(\Gamma^\ast)}{\lambda_j(\Gamma^\ast)} = \max_{\Gamma \in \mathcal G_{G,\alpha}} \frac{\lambda_k(\Gamma)}{\lambda_j(\Gamma)}, \qquad
        \frac{\lambda_k(\Gamma_\ast)}{\lambda_j(\Gamma_\ast)} = \min_{\Gamma \in \mathcal G_{G,\alpha}} \frac{\lambda_k(\Gamma)}{\lambda_j(\Gamma)}.
    \end{displaymath}
    \item Now take any two choices of standard and Dirichlet conditions, corresponding to two vectors $\alpha,\beta$ each of whose entries take on only the values $0$ and $\infty$. For $\Gamma \in \GClass$, denote by $\lambda_k$ and $\tau_k$ the $k$-th eigenvalue of the Laplacian on $\Gamma$ with the vertex conditions corresponding to $\alpha$ and $\beta$, respectively. Then for any $k,j \geq 1$ (or $j \geq 2$ if $\beta = 0$, i.e. for the $\tau_k$ only standard vertex conditions are present), there exist $\Gamma^\ast, \Gamma_\ast \in \mathcal G_G$ such that
    \begin{displaymath}
        \frac{\lambda_k(\Gamma^\ast)}{\tau_j(\Gamma^\ast)} = \max_{\Gamma \in \mathcal G_G} \frac{\lambda_k(\Gamma)}{\tau_j(\Gamma)}, \qquad
        \frac{\lambda_k(\Gamma_\ast)}{\tau_j(\Gamma_\ast)} = \min_{\Gamma \in \mathcal G_G} \frac{\lambda_k(\Gamma)}{\tau_j(\Gamma)}.
    \end{displaymath}
    \end{enumerate}
\end{proposition}

Note that $\Gamma^\ast$ and $\Gamma_\ast$ are certainly not unique, since any homethetic scaling of any extremizer will continue to be an extremizer. Also note that they may, a priori, have one or more edge lengths equal to zero, when considered to have $G$ as an underlying discrete graph; in this case, the correct vertex conditions are determined as described in Definition~\ref{def:underlying-vc}.

\begin{proof}[Proof of Proposition~\ref{prop:existence-optimizers}]
We first consider (1) and the case of the maximizers. Since the ratio $\frac{\lambda_k}{\lambda_j}$ is scale invariant, we may without loss of generality restrict to considering all graphs $\Gamma \in \mathcal G_{G_\alpha}$ with total length exactly $1$. In this case, note that
\begin{displaymath}
    \lambda_j(\Gamma) \geq \frac{\pi^2}{4}
\end{displaymath}
due to Nicaise' Faber--Krahn-type inequality \cite[Th\'eor\`eme~3.1]{N87}. (More precisely, if $j=1$ and there is at least one Dirichlet vertex, we have this bound; if $j \geq 2$, then even with no Dirichlet vertices the bound may be improved to $\pi^2$.) We can similarly bound $\lambda_k (\Gamma)$ from above, for example using \eqref{eq:general-graph-dnb} (and noting that $\beta$, $N$ and $D$ may each be replaced by $m = \# V(G)$, say).

In particular, within the given class the ratio $\frac{\lambda_k}{\lambda_j}$ is bounded from above (and naturally from below by $1$), and thus there exist both a minimizing and a maximizing sequence.

Consider the case of the latter, calling the sequence $\Gamma_n$. Since each $\Gamma_n$ has the same underlying discrete graph, each $\Gamma_n$ can be specified uniquely by a vector of edge lengths $(\ell_{1,n},\ldots,\ell_{p,n}) \in \R_+^p$ (where here $p = \# E(G)$ is the common number of edges of the graphs, noting that $\ell_{i,n}$ may be zero).

Up to a subsequence, since each $\ell_{i,n} \in [0,1]$, we have convergence of each edge length to a limit value. Now the limit graph $\Gamma^\ast$ will still belong to $\mathcal G_G$, although some edge lengths may be zero. The continuity result of \cite{BLS19} in the form of Theorem~\ref{thm:edge-length-continuity}(1) implies that both $\lambda_k(\Gamma_n)$ and $\lambda_j(\Gamma_n)$ converge; thus $\Gamma^\ast$ is indeed a maximizer in the given class.

The case of the minimizing sequence, and the proof of (2), are completely analogous.
\end{proof}

\section{On Ashbaugh--Benguria-type bounds for Dirichlet trees}
\label{sec:ab-dirichlet-trees}

The proof of Theorem~\ref{thm:vc} will be via induction on the number $m \geq 1$ of edges of $\Gamma$. If $m=1$ (or $2$) there is nothing to prove, since $\Gamma$ is an interval.

The scheme of the proof of the induction step is roughly as follows: we will show by contradiction that any graph $\Gamma$ on $m$ edges
which is a local maximizer (with respect to edge length perturbations) of the ratio $\frac{\lambda_1}{\tau_1}$ must have at least one edge of length zero, and thus be identical to a graph on at most $m-1$ edges, for which the statement is true by the induction hypothesis.

To show this, we will assume $\gamaste$ is any local critical point and find a local perturbation of $\gamaste$, essentially using a surgery argument, whose ratio is larger, and which has the same topology as $\gamaste$ up to setting some of the edge lengths to be equal to zero. More precisely, we will replace a pendant star subgraph (see Definition \eqref{def:pendant-star}) of $\gamaste$ with an interval.

The proof will be divided into several lemmas. In Lemma~\ref{lem:critical-point-equal-amplitudes} and~\ref{lem:critical-point-equal-lengths} we will describe certain necessary conditions on the edge lengths for the graph $\gamaste$ to be a critical point, based on the Hadamard-type formula in Theorem~\ref{thm:edge-length-continuity}. We will then show (Lemma~\ref{lem:interval-star-dirichlet-robin-comparison}) that, if we cut a suitable pendant star from the rest of $\gamaste$ along one of the eigenfunctions, replacing it by an interval of the right length will leave the Dirichlet eigenvalue $\lambda_1$ unaffected but lower the Neumann eigenvalue $\tau_1$. A surgery principle from \cite{BKKM19} will allow us to complete the proof by gluing the interval back to the rest of $\gamaste$ in place of the star.

We now go into the details. Suppose the statement is true for all graphs with at most $m-1 \geq 1$ edges (and all possible choices of Neumann leaf for all possible graphs). Now let $\Gamma$ be any tree with $m$ edges, assumed all to have strictly positive length, and choose any leaf $v$ of $\Gamma$ to be equipped with the Neumann condition. Let $G$ be the (proper) underlying discrete graph associated with $\Gamma$ (which by construction will also have $m$ edges), so that $\Gamma \in \GClass$. By Proposition~\ref{prop:existence-optimizers}(2) there exists some $\gamaste \in \GClass$ such that
\begin{equation}
\label{eq:gamma-ast-max-ratio}
    \frac{\lambda_1(\Gamma^\ast)}{\tau_1(\Gamma^\ast)} = \max_{\Gamma' \in \GClass} \frac{\lambda_1(\Gamma')}{\tau_1(\Gamma')}
\end{equation}
(under the convention that the Neumann leaf always corresponds to the same vertex in $G$).

We will show that $\gamaste$ can have at most $m-1$ edges of nonzero length. This, together with the induction hypothesis, will then imply that
\begin{equation*}
%\label{eq:ainda-nao-gamaste}
    \frac{\lambda_1(\Gamma)}{\tau_1(\Gamma)} < \frac{\lambda_1(\gamaste)}{\tau_1(\gamaste)} \leq 4,
\end{equation*}
where the strict inequality comes from the fact that $\Gamma$, having $m$ edges of positive length, cannot be a maximizer. (Note that our argument will only use the fact that $\gamaste$ is an interior critical point with respect to the edge length vector; it does not have to be a global maximizer for this graph topology.)

Suppose that $\gamaste$ in fact has all $m$ edges of positive length. Since $m\geq 3$, upon inserting a dummy vertex if necessary,
we can find a pendant star
\begin{equation}
\label{eq:pendant-star}
    {\PS} \subset \gamaste
\end{equation}
of $\gamaste$ with the following properties:

\begin{definition}
\label{def:pendant-star}
    For the remainder of this section, we will call a subgraph ${\PS}$ of a given tree $\Gamma$ a \emph{pendant star} if, as a graph, it is a star with at least three leaves, and there exists a point $w \in \Gamma$, taken to be a dummy vertex (i.e. of degree two, artificially inserted in the interior of an edge if necessary), such that, treating ${\PS}$ as a (closed) subset of $\Gamma$,
    \begin{displaymath}
        {\PS} \cap \overline{\Gamma \setminus {\PS}} = \{w\}.
    \end{displaymath}
    Unless otherwise stated, we will assume ${\PS}$ is equipped with a Dirichlet condition at all leaves except for $w$, that is, all leaves which, as vertices of $\Gamma$, are also Dirichlet leaves.
\end{definition}

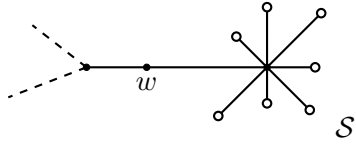
\begin{figure}[H]
\center
\begin{tikzpicture}[scale=0.8]
%\coordinate (a) at (0,0);
\coordinate (c) at (2,0);
\coordinate (b) at (4,0);
\draw[fill] (c) circle (1.5pt);
%\draw[thick, dashed] (a) -- (c);
\draw[thick] (c) -- (b);
\draw[fill] (b) circle (1.5pt);
\draw[thick] (b) -- (4.8,0);
\draw[thick] (b) -- (4,1);
\draw[thick] (b) -- (4,-0.6);
\draw[thick] (b) -- (4.9,0.9);
\draw[thick] (b) -- (4.71,-0.71);
\draw[thick] (b) -- (3.49,0.51);
\draw[thick] (b) -- (3.19,-0.81);
\draw[thick,fill=white] (4.8,0) circle (2pt);
\draw[thick,fill=white] (4,1) circle (2pt);
\draw[thick,fill=white] (4,-0.6) circle (2pt);
\draw[thick,fill=white] (4.9,0.9) circle (2pt);
\draw[thick,fill=white] (4.71,-0.71) circle (2pt);
\draw[thick,fill=white] (3.49,0.51) circle (2pt);
\draw[thick,fill=white] (3.19,-0.81) circle (2pt);
\draw[thick] (1,0) -- (c);
\node at (c) [anchor=north] {$w$};
\draw[fill] (1,0) circle (1.5pt);
\draw[thick,dashed] (1,0) -- (0.1,0.7);
\draw[thick,dashed] (1,0) -- (-0.3,-0.5);
\node at (5.3,-1) {${\PS}$};
\end{tikzpicture}
\caption{A pendant star ${\PS}$ attached to a tree at the dummy vertex $w$.}
\label{fig:pendant-star}
\end{figure}

Note in particular that there is a unique edge in $\Gamma$ (the edge corresponding to, or containing, the dummy vertex $w$) which links ${\PS}$ to the rest of $\Gamma$. We will always assume that our pendant star ${\PS}$ does \emph{not} contain the distinguished leaf of $\Gamma$ which is equipped with a Neumann condition for $\tau_1$, that is, all leaves of $\Gamma$ which are also leaves of ${\PS}$ (of which we require that there be at least two), should in fact be equipped with a Dirichlet condition in both.

To simplify the next steps of the proof, we will use the setup of Theorem~\ref{thm:edge-length-continuity}: For a given discrete graph $G$ with $n$ edges, we consider all graphs $\Gamma (\ell) \in \mathcal{G}_G$ with underlying discrete graph $G$ and edge length vector $\ell := (\ell_1,\ldots,\ell_n)$, with $\ell_i$ being the length of the $i$-th edge $e_i$. We write $\lambda_1(\ell):=\lambda_1(\Gamma(\ell))$ and $\tau_1(\ell):=\tau_1(\Gamma(\ell))$ for the respective eigenvalues, $c_{D,i}$ for the amplitude of the eigenfunction associated with $\lambda_1(\ell)$ on the edge $\ell_i$ and $c_{N,i}$ the amplitude of the eigenfunction associated with $\tau_1(\ell)$, where both are assumed to have $L^2$-norm $1$ on $\Gamma(\ell)$. Additionally, we also denote the frequencies by $k_D = \sqrt{\lambda_1}$ and $k_N = \sqrt{\tau_1}$ (we will omit the dependence on $\Gamma$ and $\ell$ when there is no danger of confusion).

\begin{lemma}
\label{lem:critical-point-equal-amplitudes}
Let $\Gamma$ be any graph. Suppose the edge lengths, all assumed positive, are such that the map $\R^m \ni \ell \mapsto \frac{\lambda_1(\ell)}{\tau_1(\ell)}$ is at a critical point. Then the amplitudes $c_{D,i}$ and $c_{N,i}$ are equal for every $i = 1,...,m$, i.e. there exists $c_i > 0$ such that 
for all $i = 1,...,n$, 
$c_{D,i} = c_{N,i} =: c_i$ .
\end{lemma}

\begin{proof}%[Proof of Lemma~\ref{lem:critical-point-equal-amplitudes}.]
  By assumption, $\frac{\partial }{\partial \ell_i} \frac{\lambda_1}{\tau_1} = 0$ for all $i = 1,..., n$, which implies that
  \begin{equation}
    \label{eq:balancing-condition-ratio}
     \frac{\frac{\partial \lambda_1}{\partial \ell_i}}{\lambda_1} = \frac{\frac{\partial \tau_1}{\partial \ell_i}}{\tau_1}
  \end{equation}
  for all $i$. Denoting by $\psi_1 \sim \lambda_1$ and $\phi_1 \sim \tau_1$ the respective eigenfunctions, chosen positive and normalized to have $L^2$-norm 1, and associating $e_i \sim [0,\ell_i]$ we can write
  \begin{displaymath}
     \psi_1|_{e_i}(x) = c_{D,i} \sin(k_D(x+\theta_{D,i}))
  \end{displaymath}
  for all $i = 1,...,n$, and likewise
  \begin{displaymath}
      \phi_1|_{e_i}(x) = c_{N,i}\sin(k_N(x+\theta_{N,i})).
  \end{displaymath}
  Applying Theorem ~\ref{thm:edge-length-continuity}, we obtain
  \begin{displaymath}
      \frac{\partial \lambda_1}{\partial \ell_i} = -\lambda_1c_{D,i}^2\cos^2(k_D(x+\theta_{D,i}))-\lambda_1c_{D,i}^2\sin^2(k_D(x+\theta_{D,i})) = -\lambda_1c_{D,i}^2
  \end{displaymath}
  and likewise
  \begin{displaymath}
      \frac{\partial \tau_1}{\partial \ell_i} = -\tau_1c_{N,i}^2
  \end{displaymath}
for all $i = 1,...,n$. Combining this with \eqref{eq:balancing-condition-ratio} and noting that the amplitudes are positive, we get that
  \begin{displaymath}
      c_{D,i} = c_{N,i} \eqqcolon c_i
  \end{displaymath}
  for all $i$.
\end{proof}

\begin{lemma}
\label{lem:critical-point-equal-lengths}
Under the assumptions of Lemma \ref{lem:critical-point-equal-amplitudes}, let ${\PS}$ be a pendant star of $\Gamma$ (in the sense of Definition \ref{def:pendant-star}), with $n \geq 2$ Dirichlet leaves of lengths $\ell_1,\ell_2,...,\ell_n$. Then there exists $\ell > 0 $ such that $\ell_i = \ell$ for all $i = 1,..., n$. Moreover, the frequencies $k_D$ and $k_N$ satisfy the relation $k_D + k_N = \frac{\pi}{\ell}$. 
\end{lemma}

The proof requires a careful analysis of the eigenfunctions on the pendant star and will be given in Appendix~\ref{sec:appendix}.

\begin{remark}
Note that the relation $k_D + k_N = \frac{\pi}{\ell}$ for ${\PS}$ implies that for any other pendant star $\widetilde{{\PS}}$ of $\Gamma$ (not containing the distinguished vertex $v$), the lengths of the Dirichlet leaves must also be equal to the same constant $\ell$.
\end{remark}

The next lemma is a key surgery result which recalls how Ashbaugh--Benguria compare $\Omega$ with a ball of the same first Dirichlet eigenvalue \emph{rather than a ball of the same volume} (see the description of the strategy in and around \cite[eq.~(2.2.7)]{A00}). We first need to introduce a bit more notation. Let $k_\alpha^I$ be the square root of the first eigenvalue $\lambda_1^I(\alpha)$ of an interval $I$ with a Dirichlet condition at one endpoint and a Robin condition with parameter $\alpha \in \R$ at the other, and $k_\alpha^S$ be the square root of the first eigenvalue $\lambda_1^{S}(\alpha)$ of (any) star graph $S$ with one leaf equipped with a Robin condition with parameter $\alpha$ and all other edges having the same length $\ell$ and a Dirichlet condition at their leaf.

\begin{lemma}
\label{lem:interval-star-dirichlet-robin-comparison}
Let $\alpha_D,\alpha_N \in \mathbb{R}$ be any two real numbers such that $k_{\alpha_D}^{S} > k_{\alpha_N}^{{S}} > 0$ and $k_{\alpha_D}^{S} + k_{\alpha_N}^{S} = \frac{\pi}{\ell}$, where $\ell$ is the length of the Dirichlet leaves of the star ${S}$. Then there exists an interval $I$ such that $\lambda_1^I(\alpha_D) = \lambda_1^{S}(\alpha_D)$ and
\begin{equation}
\label{eq:interval-star-dirichlet-robin-comparison}
    %\frac{k_{\alpha_D}^I}{k_{\alpha_N}^I} > \frac{k_{\alpha_D}^S}{k_{\alpha_D}^{\PS}} 
    \lambda_1^I(\alpha_N) < \lambda_1^{{S}}(\alpha_N).
\end{equation}
%with the inequality being strict.
\end{lemma}

As with the proof of the previous lemma, the proof here requires a detailed study of the respective eigenfunctions using a number of trigonometric identities and relations, and will be given in Appendix~\ref{sec:appendix}.

%With these technical results to hand, we can now complete the proof of Theorem~\ref{thm:vc}. We will divide it into two parts: the proof of the actual inequality, and afterwards the proof that there is equality if and only if $\Gamma$ is an interval.

\begin{proof}[Conclusion of the proof of Theorem~\ref{thm:vc}]
We recall that we need to complete the induction step on the number of edges $m\geq 3$; to this end we return to assuming that $\gamaste$ is the tree satisfying \eqref{eq:gamma-ast-max-ratio}, and that $\gamaste$ has $m$ edges of positive length. This means in particular that Lemmas~\ref{lem:critical-point-equal-amplitudes} and~\ref{lem:critical-point-equal-lengths} apply, the latter to the pendant star ${\PS}$ described at the beginning of the proof (see \eqref{eq:pendant-star} and Definition~\ref{def:pendant-star}), with $n\geq 2$ Dirichlet leaves of the same length $\ell>0$ by Lemma~\ref{lem:critical-point-equal-lengths}.

As before, let $\psi_1 \sim \lambda_1(\gamaste)$ and $\phi_1 \sim \tau_1(\gamaste)$ be the two eigenfunctions, normalized appropriately. We take $\alpha_N, \alpha_D \in \R$ to be such that $\psi_1|_{{\PS}}$, respectively $\phi_1|_{{\PS}}$, is the first eigenfunction on the star with a Robin condition of strength $\alpha_D$, respectively $\alpha_N$, at the vertex $w$ which separates ${\PS}$ from the rest of $\gamaste$, and Dirichlet conditions at the other $n \geq 2$ leaves (see Figure~\ref{fig:pendant-star}).

For an interval $I$ and $\alpha \in \R$, as before denote by $\lambda_1^I(\alpha)$ the first eigenvalue of the Laplacian with a Dirichlet condition at one endpoint and a Robin condition of strength $\alpha$ at the other; we will likewise use $\lambda_1^{{\PS}}(\alpha)$ for the first eigenvalue on ${\PS}$ with a Robin condition at $w$ and a Dirichlet condition at all other leaves.

We let $I$ be the interval such that $\lambda_1^I(\alpha_D) = \lambda_1(\gamaste)$. Then by Lemma~\ref{lem:interval-star-dirichlet-robin-comparison} and our choice of $\alpha_D$, $\alpha_N$, we know that $\lambda_1^I(\alpha_D) = \lambda_1^{{\PS}}(\alpha_D)$ and $\lambda_1^I (\alpha_N) < \lambda_1^{{\PS}}(\alpha_N)$. We also denote by $\psi_1^I$ the eigenfunction on $I$ associated with $\lambda_1^I(\alpha_D)$ and by $\phi_1^I$ the eigenfunction associated with $\lambda_1^I(\alpha_N)$ (both chosen positive and with $L^2$-norm $1$).

We now construct a new graph $\widehat\Gamma$ from $\gamaste$ by deleting ${\PS}$ at $w$ and gluing $I$ (where we imagine that the Robin vertex of $I$ will correspond to $w$). Note that $\widehat\Gamma \in \mathcal{G}_{\gamaste}$, since $\widehat\Gamma$ can be obtained by shrinking the Dirichlet edges of ${\PS}$ to $0$, and altering the length of the remaining edge if necessary. We claim that
\begin{equation}
\label{eq:gamaste-um-chapeu}
    \frac{\lambda_1(\widehat\Gamma)}{\tau_1(\widehat\Gamma)}
    > \frac{\lambda_1(\gamaste)}{\tau_1(\gamaste)}.
\end{equation}
We note that \eqref{eq:gamaste-um-chapeu} will complete the proof of Theorem~\ref{thm:vc}, since we will have obtained a contradiction to the assumption that the maximizing graph for the ratio $\frac{\lambda_1}{\tau_1}$ within $\mathcal{G}_{\gamaste}$ had $m$ edges of strictly positive length: the maximum can only be attained at a graph with at most $m-1$ edges. In particular, for any graph with $m$ edges of positive length $\Gamma$, necessarily
\begin{equation}
\label{eq:ainda-nao-gamaste}
    \frac{\lambda_1(\Gamma)}{\tau_1(\Gamma)} < 4.
\end{equation}
 This means in particular that interval is the unique maximizer, since \eqref{eq:ainda-nao-gamaste} states directly that any tree with $m \geq 3$ edges (of nonzero length) has ratio strictly less than $4$.

The inequality \eqref{eq:gamaste-um-chapeu} will follow from a surgery argument. We first claim that $\lambda_1(\widehat\Gamma) = \lambda_1(\gamaste)$; to see this we consider the function
\begin{displaymath}
    \hat\psi_1(x) := \begin{cases} a\psi_1^I (x) \qquad &\text{if } x \in I,\\ \psi_1(x) \qquad &\text{if } x \in \widehat\Gamma\setminus I = \gamaste \setminus {\PS},\end{cases}
\end{displaymath}
where the constant $a>0$ is chosen to ensure that $\hat\psi_1$ is continuous at the vertex $w$ at which $I$ is glued to $\widehat\Gamma \setminus I$, and thus on $\widehat\Gamma$.

Now, since by our choice of the length of $I$ and of $\alpha_D$, $\hat\psi_1$ satisfies the same eigenvalue equation $-\hat\psi_1=\lambda_1(\gamaste)\hat\psi_1$ on both $I$ and $\widehat\Gamma\setminus I$, as well as continuity and the Kirchhoff condition at the vertex $w$, and continuity and Kirchhoff conditions at all other non-leaf vertices of $\widehat\Gamma$ (since $\psi_1$ and $\psi_1^I$ do), necessarily $\hat\psi_1$ is the first eigenfunction, with eigenvalue $\lambda_1(\gamaste)$, of $\widehat\Gamma$ with Dirichlet conditions at all leaves of $\widehat\Gamma$, as well as continuity and Kirchhoff conditions elsewhere. That is, $\lambda_1(\widehat\Gamma) = \lambda_1(\gamaste)$, as claimed.

Hence, to prove \eqref{eq:gamaste-um-chapeu} it suffices to prove the inequality $\tau_1(\widehat\Gamma)<\tau_1(\gamaste)$. To this end we denote by $\lambda_1^{\alpha}(\gamaste \setminus {\PS})$ the first eigenvalue of the Laplacian on $\gamaste\setminus{\PS}$ with the same vertex conditions as for $\tau_1$ except that at $w$ (the vertex where we cut ${\PS}$ from the rest of the graph) we impose a Robin potential of strength $-\alpha_N \in \R$ so that
\begin{displaymath}
    \tau_1 (\gamaste) = \lambda_1^{\gamaste \setminus {\PS}}(-\alpha_N),
\end{displaymath}
and $\phi_1$ continues to be the corresponding eigenfunction on $\gamaste \setminus {\PS}$ (cf.\ \cite[eq.~(3.2)]{BKKM19}). We note that, similarly,
\begin{displaymath}
    \tau_1 (\gamaste) = \lambda_1^{{\PS}}(\alpha_N)
\end{displaymath}
precisely, by the choice of $\alpha_N$.

Now exactly the right conditions are satisfied to use the surgery principle in \cite[Theorem~3.10(1)]{BKKM19} (attaching a pendant graph at a vertex, with $k=r=1$, $\mathcal{G} = \gamaste \setminus {\PS}$, $\mathcal{H} = I$ and the attachment point at the vertex $w$): Since
\begin{displaymath}
    \lambda_1^I(\alpha_N) < \lambda_1^{{\PS}}(\alpha_N) = \tau(\gamaste) = \lambda_1^{\gamaste \setminus {\PS}}(-\alpha_N)
\end{displaymath}
(where the first inequality follows exactly from Lemma~\ref{lem:interval-star-dirichlet-robin-comparison}), we obtain
\begin{displaymath}
    \tau_1(\widehat\Gamma) < \tau_1(\gamaste),
\end{displaymath}
since $\widehat\Gamma$ is the graph obtained by gluing $I$ to $\gamaste \setminus {\PS}$ at $w$ and the condition imposed at $w$ will be a Robin condition with potential of strength $\alpha_N + (-\alpha_N) = 0$, i.e. we have continuity and Kirchhoff conditions at $w$.

This completes the proof of \eqref{eq:gamaste-um-chapeu} and hence of Theorem~\ref{thm:vc}.
\end{proof}

%\begin{proof}[Proof of the characterization of equality in Theorem~\ref{thm:vc}]
%\end{proof}

We finish this section by deriving Theorem~\ref{thm:ashbaugh-benguria} from the result we have just proved. The basic idea is to apply the latter to the \emph{Neumann domains} of the eigenfunction associated with $\lambda_1(\Gamma)$, that is, we cut apart $\Gamma$ at a point where the eigenfunction $\psi_1$, chosen positive, reaches its maximum.

\begin{proof}[Proof of Theorem~\ref{thm:ashbaugh-benguria} based on Theorem~\ref{thm:vc}]
    We start by proving the inequality. Let $v$ be any point, without loss of generality a vertex, at which the first eigenfunction $\psi_1$ of $\Gamma$, chosen positive attains a local maximum. This divides the graph $\Gamma$ into some number $n \geq 2$ of connected components $\Gamma_i$, $i=1,\ldots,n$ (formally, these are the graphs taken as the closures of the connected components of $\Gamma \setminus \{v\}$).

    Now since $\psi_1$, restricted to any $\Gamma_i$, satisfies the eigenvalue equation edgewise and the corresponding vertex conditions vertex-wise, we see it is an eigenfunction of the Laplacian on each $\Gamma_i$, with a Dirichlet condition at all leaves except the new leaf created at $v$, where it satisfies a Neumann condition. Since it is positive everywhere, a standard argument shows that it must be the first eigenfunction for this problem, thus in particular
    \begin{displaymath}
        \lambda_1(\Gamma) = \tau_1(\Gamma_i)
    \end{displaymath}
    for all $i=1,\ldots,n$.

    Now by Theorem~\ref{thm:vc} applied to each $\Gamma_i$, we know that
    \begin{displaymath}
        \frac{\lambda_1(\Gamma_i)}{\tau(\Gamma_i)} \leq 4
    \end{displaymath}
    for $i=1,\ldots,n$, and in particular for $i=1,2$. But on the other hand the inequality
    \begin{displaymath}
        \lambda_2(\Gamma) \leq \max\{ \lambda_1(\Gamma_1),\lambda_1(\Gamma_2)\}
    \end{displaymath}
follows from the variational characterization of $\lambda_2$:  Consider a pair of $\Gamma_i$, without loss of generality $i = 1,2$,
and let the eigenfunctions associated with $\lambda_1(\Gamma_i)$, extended by $0$ to the entire graph $\Gamma$ be denoted $\varphi_{1,2}$.  Choose $d_i$ so that $\varphi_\ast = d_1 \varphi_1 + d_2 \varphi_2 \in H_0^1(\Gamma, V_D)$ is normalized in $L^2$ and orthogonal to $\psi_1$.  It follows that
\begin{displaymath}
\begin{aligned}
\lambda_2 &\le \int_\Gamma{|\varphi_*'|^2}\\
&=\int_{\Gamma_1}{|d_1 \varphi_1'|^2}+
\int_{\Gamma_2}{|d_2 \varphi_2'|^2} \\
&\le |d_1|^2 \lambda_1(\Gamma_1) \int_{\Gamma_1}{|\varphi_1|^2} +
|d_2|^2 \lambda_1(\Gamma_2) \int_{\Gamma_2}{|\varphi_2|^2}\\
&\le \max(\lambda_1(\Gamma_1), \lambda_1(\Gamma_2) \int_\Gamma{|\varphi_*|^2}) = \max(\lambda_1(\Gamma_1), \lambda_1(\Gamma_2)).
\end{aligned}
\end{displaymath}

Hence
    \begin{displaymath}
        \frac{\lambda_2(\Gamma)}{\lambda_1(\Gamma)} \leq \max_i \left\{ \frac{\lambda_1(\Gamma_i)}{\tau_1(\Gamma_i)}\right\} \leq 4.
    \end{displaymath}

    We now consider the slightly more delicate case of equality. Keeping the notation from the above proof of the inequality, we suppose that $\frac{\lambda_2(\Gamma)}{\lambda_1(\Gamma)} = 4$.

    Take any pair of subgraphs, say $\Gamma_1$, $\Gamma_2$, then the above argument implies that
    \begin{displaymath}
        4 = \frac{\lambda_2(\Gamma)}{\lambda_1(\Gamma)} \leq \max_i \left\{ \frac{\lambda_1(\Gamma_i)}{\tau_1(\Gamma_i)}\right\} \leq 4,
    \end{displaymath}
    meaning there is equality everywhere. Suppose without loss of generality that the maximum is attained by $\Gamma_1$; then the characterization of equality in Theorem~\ref{thm:vc} implies that $\Gamma_1$ is an interval.

    But now note the equality $\max \{\lambda_1(\Gamma_1),\lambda_1(\Gamma_2)\} = \lambda_2(\Gamma)$; more precisely, the two eigenfunctions $\psi_{1,1}$ and $\psi_{1,2}$ associated, respectively, with $\lambda_1(\Gamma_1)$ and $\lambda_1(\Gamma_2)$ (and extended by zero to form test functions on $\Gamma$), yield equality when used as test functions in the variational characterization of $\lambda_2(\Gamma)$. This implies that there must be an eigenfunction $\psi_2$ associated with $\lambda_2(\Gamma)$ which is a linear combination of them, that is, $\psi_2 = a_1 \psi_{1,1} + a_2 \psi_{1,2}$ for some $a_1, a_2 \in \R$, see \cite[Lemma~4.1(1)]{BKKM19}. In fact, necessarily $a_1,a_2 \neq 0$, since $\psi_2$ must change sign but $\psi_{1,i}$, as the respective first eigenfunctions on $\Gamma_i$, do not.

    But this then also implies that $\lambda_1(\Gamma_1) = \lambda_1(\Gamma_2) = \lambda_2 (\Gamma)$, from which we may deduce that $\Gamma_2$ is also maximizing for the ratio $\frac{\lambda_1}{\tau_1}$ (that is, $\frac{\lambda_1(\Gamma_2)}{\tau_1 (\Gamma_2)} = 4$ as well). Hence, once again by the characterization of equality in Theorem~\ref{thm:vc}, $\Gamma_2$ must also be an interval. But since the first eigenvalues are equal, the intervals must have the same length as each other.

    Finally, if the number $n$ of connected subgraphs meeting at $v$ is more than two, we may now repeat this argument inductively, at each step considering the pair $\Gamma_1$ and $\Gamma_i$, to conclude that $\Gamma_i$ is likewise an interval of the same length as $\Gamma_1$. We can thus conclude, finally, that $\Gamma$ is an equilateral star with $n$ edges.
\end{proof}

\section{General bounds for Dirichlet trees}
\label{sec:general-dirichlet-trees}

\begin{proof}[Proof of Theorem~\ref{thm:dirichlet-tree-2k-k}]
Fix a compact Dirichlet tree $\Gamma$. Assume for the present that there exists an eigenfunction $\psi_k$ associated with $\lambda_k(\Gamma)$ which has (at least) $k$ nodal domains, that is, that there are (at least) $k$ connected components of $\{ x \in \Gamma: \psi_k(x) \neq 0\}$. Denote their respective closures by $\Gamma_1,\ldots,\Gamma_k$, which we may treat as subgraphs of $\Gamma$. Then a standard result states that, for all $i=1,\ldots,k$,
\begin{displaymath}
    \lambda_1(\Gamma_i) = \lambda_k (\Gamma),
\end{displaymath}
where on $\Gamma_i$ we impose Dirichlet conditions exactly at the leaves of $\Gamma_i$, including on $\partial\Gamma_i$; to see this, observe that $\psi_k|_{\Gamma_i}$ satisfies the eigenvalue equation $-\psi_k''=\lambda_k(\Gamma)\psi_k$ edgewise on $\Gamma_i$, as well as all relevant vertex conditions. Thus $(\lambda_k(\Gamma),\psi_k|_{\Gamma_i})$ is an eigenpair on $\Gamma_i$; since $\psi_k$, by construction, does not change sign on $\Gamma_i$, it must correspond to the first eigenvalue of $\Gamma_i$.

In particular, if we denote by $\Gamma'$ the disjoint union of the $k$ graphs $\Gamma_1,\ldots,\Gamma_k$, then
\begin{displaymath}
    \lambda_1(\Gamma') = \lambda_k(\Gamma') = \lambda_k (\Gamma),
\end{displaymath}
where this eigenvalue now has multiplicity $k$ in $\Gamma'$. Suppose without loss of generality that $\lambda_2(\Gamma_1) = \max_i \lambda_2(\Gamma_i)$.  Then
\begin{displaymath}
    \lambda_{2k}(\Gamma') \leq \lambda_2 (\Gamma_1),
\end{displaymath}
since $\lambda_{2k}(\Gamma')$ cannot be larger than the largest of the $2k$ eigenvalues (counted with multiplicities) $\lambda_1(\Gamma_1),\ldots,\lambda_1(\Gamma_k),\lambda_2(\Gamma_1),\ldots,\lambda_2(\Gamma_k)$ (by the same variational argument as in the proof of Theorem~\ref{thm:ashbaugh-benguria}).

On the other hand, $\lambda_{2k}(\Gamma) \leq \lambda_{2k}(\Gamma')$, since $\Gamma'$ was created from $\Gamma$ via the insertion of additional Dirichlet conditions. Hence, putting everything together,
\begin{displaymath}
    \frac{\lambda_{2k}(\Gamma)}{\lambda_k(\Gamma)} \leq \frac{\lambda_{2k}(\Gamma')}{\lambda_1(\Gamma')} \leq \frac{\lambda_2(\Gamma_1)}{\lambda_1(\Gamma_1)}.
\end{displaymath}
Now since $\Gamma_1$ is itself a compact Dirichlet tree, Theorem~\ref{thm:ashbaugh-benguria} implies that
\begin{displaymath}
    \frac{\lambda_{2k}(\Gamma)}{\lambda_k(\Gamma)} \leq \frac{\lambda_2(\Gamma_1)}{\lambda_1(\Gamma_1)} \leq 4.
\end{displaymath}

It remains to the consider the case where $\psi_k$ has fewer than $k$ nodal domains. Here we can use a now quite standard perturbation argument. More precisely, given any compact tree $\Gamma$, with (proper) underlying discrete graph $G$, it is known that the set of edge length vectors ${\bf l}$ for which all eigenvalues of the Dirichlet Laplacian on the metric tree graph $(G,{\bf l})$ are simple, and the $k$-th eigenfunction has exactly $k$ nodal domains (more precisely: its zeros do not coincide with any vertices of the graph, and there are exactly $k-1$ of them) is of the second Baire category in $(\R^m)_+$ (where $m$ is the number of edges of $G$). This follows directly from \cite[Remark~2.4 and Theorem~2.5]{B08}
with $\beta=0$.

Hence, in particular, there exists a sequence of compact Dirichlet trees $\Gamma_n$ for which the corresponding edge length vectors ${\bf l}_n \to {\bf l}$, where ${\bf l}$ is the edge length vector for $\Gamma$, and for which, for all $n$, $\psi_{k}^{(n)}$, the (unique up to scalar multiples) $k$-th eigenfunction on $\Gamma_n$, has exactly $k$ nodal domains; in particular,
\begin{displaymath}
    \frac{\lambda_{2k}(\Gamma_n)}{\lambda_k(\Gamma_n)} \leq 4
\end{displaymath}
by what was shown above. But now the continuity result in Theorem~\ref{thm:edge-length-continuity} implies that the same bound must be true in the limit, that is, for $\Gamma$.
\end{proof}

\begin{proof}[Proof of Corollary~\ref{cor:dirichlet-tree-k-j}]
For the first inequality in \eqref{eq:ab-kj}: Given $1 \leq j < k$, let $m \in \N$ be such that $2^{m-1}j < k \leq 2^m j$, that is, $m = \lceil \log_2 \frac{k}{j}\rceil$; then applying Theorem~\ref{thm:dirichlet-tree-2k-k} iteratively $m$ times yields
\begin{displaymath}
    \frac{\lambda_k (\Gamma)}{\lambda_j(\Gamma)} \leq \frac{\lambda_{2^m j}(\Gamma)}{\lambda_j(\Gamma)}\leq 4^m.
\end{displaymath}

For the second inequality in \eqref{eq:ab-kj}, that is, $4^{\lceil\log_2 \frac{k}{j}\rceil} \leq 4 \cdot \frac{k^2}{j^2}$: after some basic manipulations (take square roots and then logarithms of both sides), we see this will follow if $\lceil\log_2 x\rceil \leq \log_2(2x)$ for all $x>0$. But this latter inequality is immediate: given $x>0$, choose $m \in \Z$ such that $x \in (2^{m-1},2^m]$, then $\lceil\log_2 x\rceil = m$, while $\log_2(2x) \geq \log_2(2\cdot2^{m-1}) = m$.
\end{proof}

\begin{remark}
    As noted in Remark~\ref{rem:big-ab}(2), it is not clear what the optimal constant in \eqref{eq:ab-kj} should be. However, the only step in the proof of Corollary~\ref{cor:dirichlet-tree-k-j} which is not sharp (other than the step $4^{\lceil\log_2 \frac{k}{j}\rceil} \leq 4\cdot\frac{k^2}{j^2}$) is the estimate $\lambda_k < \lambda_{2^{m}j}$.
\end{remark}

\section{Bounds and optimizers for general graphs}
\label{sec:general-graphs}

We start with the simple proof of Proposition~\ref{prop:ratio-inf}.

\begin{proof}[Proof of Proposition~\ref{prop:ratio-inf}]
    It is immediate that the infimum must be at least $1$. For $k > j \geq 2$ we know that if $\Gamma_k$ is an equilateral $k$-star each of whose edges has length $\ell > 0$, then $\lambda_1(\Gamma_k) = \frac{\pi^2}{4\ell^2}$, while $\lambda_2(\Gamma_k) = \ldots = \lambda_k(\Gamma_k) = \frac{\pi^2}{\ell^2}$.

    For $j=1$, note that we are bounding the ratio $\frac{\lambda_k(\Gamma)}{\lambda_1(\Gamma)}$. If $\Gamma$ is connected (after removal of all Dirichlet vertices), then $\lambda_1(\Gamma)$ is simple, so that necessarily $\lambda_k(\Gamma) > \lambda_1(\Gamma)$.  Thus if $\lambda_k(\Gamma) = \lambda_1 (\Gamma)$ then necessarily $\Gamma$ must have a Dirichlet condition in the interior. It is clear that for an equilateral $k$-star $\Gamma_k$ with a Dirichlet condition at its center and all edges of length $\ell > 0$, we have $\lambda_1(\Gamma_k) = \ldots = \lambda_k(\Gamma_k) = \frac{\pi^2}{\ell^2}$.
\end{proof}

We next turn to Theorem~\ref{thm:general-graph-k-j} and as such suppose that the graph $\Gamma$ has some given number $N \geq 0$ of Neumann leaves and $\beta \geq 0$ independent cycles (as well as any number of Dirichlet leaves).

\begin{proof}[Proof of Theorem~\ref{thm:general-graph-k-j}]
    We claim that for the given graph $\Gamma$ as described above, there is a Dirichlet tree $\widetilde\Gamma$ such that
    \begin{equation}
    \label{eq:cutting-tree}
        \lambda_{k-N-\beta}(\widetilde\Gamma) \leq \lambda_k(\Gamma) \leq \lambda_k(\widetilde\Gamma),
    \end{equation}
    where the first inequality holds for all $k \geq N + \beta + 1$, and the second for all $k \geq 1$.
    
    Indeed, we first replace all $N$ Neumann leaves of $\Gamma$ by Dirichlet leaves; this creates a new (quantum) graph $\Gamma'$ (which is however equal to $\Gamma$ as a metric graph) having $\beta$ cycles but only Dirichlet leaves; then by \cite[Theorem~3.4]{BKKM19},
    \begin{equation}
    \label{eq:no-leaves}
        \lambda_{k-N}(\Gamma') \leq \lambda_k(\Gamma) \leq \lambda_k(\Gamma'),
    \end{equation}
    where the first inequality holds for all $k \geq N+1$ and the latter inequality holds for all $k \geq 1$.
    
    We now create $\widetilde\Gamma$ out of $\Gamma'$ as follows: we select any $\beta$ points in the interior of edges, whose removal would turn $\Gamma'$ into a tree; this we can do since $\Gamma'$, like $\Gamma$, has first Betti number $\beta$. We impose Dirichlet conditions directly on these $\beta$ points and call the new graph $\widetilde\Gamma$; this is equivalent to changing the vertex condition at each of these $\beta$ points from a $0$-delta condition to an $\infty$-delta condition; as such, still by \cite[Theorem~3.4]{BKKM19},
    \begin{equation}
    \label{eq:no-cycles}
        \lambda_{k-\beta}(\widetilde\Gamma) \leq \lambda_k(\Gamma') \leq \lambda_k(\widetilde\Gamma),
    \end{equation}
    where again the first inequality holds for all $k \geq \beta + 1$ and the second holds for all $k \geq 1$.

    Combining \eqref{eq:no-leaves} and \eqref{eq:no-cycles} immediately yields \eqref{eq:cutting-tree}.

    To finish the proof, we invoke Corollary~\ref{cor:dirichlet-tree-k-j}: For any $k \geq j \geq N + \beta  + 1$ we have
    \begin{equation}
    \label{eq:general-graph-explicit}
        \frac{\lambda_k(\Gamma)}{\lambda_j(\Gamma)} \leq \frac{\lambda_k (\widetilde\Gamma)}{\lambda_{j-N-\beta}(\widetilde\Gamma)} \leq 4^{\lceil \log_2 (\frac{k}{j-N-\beta})\rceil} \leq 4 \cdot \frac{k^2}{(j-N-\beta)^2},
    \end{equation}
    where the first inequality was \eqref{eq:cutting-tree} and the second was \eqref{eq:ab-kj}.
\end{proof}

We finish with the general existence result assuming that the number of vertices is bounded from below; to that end, we assume $m \geq 2$ is fixed, as are $1 \leq j < k$ under the assumptions of Theorem~\ref{thm:existence}. The proof is quite a direct consequence of our first existence result, Proposition~\ref{prop:existence-optimizers}.

\begin{proof}[Proof of Theorem~\ref{thm:existence}]
Since the maximal number of edges is fixed, for any $\Gamma$ satisfying the assumptions of the theorem there are only finitely many possibilities for the (proper) underlying discrete graph of $\Gamma$, including the choice of vertices. In particular, the set of all graphs on at most $m$ edges, where any combination of Dirichlet and standard conditions is allowed on any graph, is equal to a finite union of sets of the form $\mathcal G_G$, where, as noted before Theorem~\ref{thm:edge-length-continuity}, we consider two discrete graphs to be different if they are equal as graphs but have different vertex conditions associated with them.

Now, by Proposition~\ref{prop:existence-optimizers}(1), in each set $\mathcal C_G$ there is a maximizer and a minimizer of the ratio $\frac{\lambda_k}{\lambda_j}$, which in particular have at most $m$ edges. Since there are only finitely many possible choices of $G$, a maximizer and a minimizer necessarily exist in the class of all graphs with at most $m$ edges.
\end{proof}

\begin{proof}[Proof of Corollary~\ref{cor:existence}]
We note at the outset that we will obtain this corollary from Proposition~\ref{prop:existence-optimizers}, essentially as a corollary of the proof of Theorem~\ref{thm:existence} and not of the actual statement of the latter.

Given bounds on $D$, $N$ and $\beta$, the total number of (non-dummy) vertices of the graph $\Gamma$ is bounded: Namely, if $\beta = 0$ and $\Gamma$ is a tree then it can have no more than $2(D+N)-1$ vertices, with equality (only) the case of a finite, rooted binary tree for which $D+N$, the total number of leaves, is a power of $2$, and the root of the tree is not counted as a vertex since it has degree two. If $\beta > 0$, then as long as $\Gamma$ is not a cycle it can be transformed into a tree $\widetilde\Gamma$ at the cost of cutting through each cycle once, in the middle of an edge, a total of $\beta$ times to produce a tree with $D$ Dirichlet leaves and $N+2\beta$ Neuman leaves, and thus at most $2(D+N+\beta)-1$ vertices, which also has $\beta$ more vertices than $\Gamma$.

Hence, as in the proof of Theorem~\ref{thm:existence}, if we restrict to those graphs with at most $D$ Dirichlet leaves, $N$ Neumann leaves and $\beta$ independent cycles, there are at most finitely many possible
underlying discrete graphs, and thus, by Proposition~\ref{prop:existence-optimizers}(1), there exist maximizers and minimizers among all graphs in the union of suitable classes of the form $\mathcal G_G$. Yet the number of cycles, along with the number of leaves of either kind, can only stay the same or decrease in the limit as one or more edge lengths shrink to zero, and hence the maximizers and minimizers among the union of the $\mathcal G_G$ still satisfy the same bounds on $D$, $N$ and $\beta$.
\end{proof}

\appendix

\section{Proofs of the lemmas}
\label{sec:appendix}

This appendix contains the proofs of the technical lemmas from Section~\ref{sec:ab-dirichlet-trees}.

\begin{proof}[Proof of Lemma~\ref{lem:critical-point-equal-lengths}]
    Before continuing, we note the relations
    \begin{displaymath}
        \frac{\pi}{\ell} > k_D > k_N > 0.
    \end{displaymath}
    The first inequality is just (strict) monotonicity with respect to domain inclusion (see, e.g., \cite[Corollary~3.12(1)]{BKKM19}, applicable since the first eigenfunction is strictly positive except at the leaves,) the second is immediate since we are replacing a Dirichlet condition with a Neumann condition (see, e.g., \cite[Theorem~3.4 or Lemma~4.1]{BKKM19}, which also implies the third inequality since with $k_N$ there is still at least one Dirichlet vertex). We will use these inequalities several times without further comment.

    Associate each Dirichlet leaf $e_i \sim [0,\ell_i]$, where $0$ corresponds to the leaf and $\ell_i$ to the central vertex. Then, using Lemma~\ref{lem:critical-point-equal-amplitudes}, the normalized eigenfunctions $\psi_1 \sim \lambda_1$ and $\phi_1 \sim \tau_1$ can be written as
    \begin{displaymath}
        \psi_1|_{e_i}(x) = c_i\sin(k_Dx)
    \end{displaymath}
    and
    \begin{displaymath}
        \phi_1|_{e_i}(x) = c_i\sin(k_Nx)
    \end{displaymath}
    for all $i = 1,...,n$.
    The continuity condition on the central vertex gives us the following system of equations:
    \begin{displaymath}
        c_i\sin(k_D\ell_i) = c_j\sin(k_D\ell_j)
    \end{displaymath}
    and
    \begin{displaymath}
        c_i\sin(k_N \ell_i) = c_j\sin(k_N \ell_j)
    \end{displaymath}
    for all $i,j = 1,...,n$, $i \neq j$. Dividing the first equation by the second, which we can do since eigenfunctions associated with the first eigenvalue do not have interior zeros, we get that
    \begin{displaymath}
        \frac{\sin(k_D \ell_i)}{\sin(k_N \ell_i)} = \frac{\sin(k_D\ell_j)}{\sin(k_N\ell_j)}
    \end{displaymath}
    for all $i,j = 1,...,n$, $i \neq j$.
    By the injectivity of the function $x \mapsto \frac{\sin(\alpha x)}{\sin(\beta x)}$ on the interval $[0,\frac{\pi}{\alpha}]$ for any $0 < \beta < \alpha$, noting that $\ell_i \in [0,\frac{\pi}{k_D}]$ for any $i$, we must have $\ell_i = \ell_j \eqqcolon \ell$ for all $i,j$. 
    Observe that this also implies $c_i = c_j \eqqcolon c$ for all $i,j$.
    
    It remains to show the relation $k_D + k_N = \frac{\pi}{\ell}$; for this we will need a careful analysis of the eigenfunctions. Let $e_0 \subset \Gamma \setminus {\PS}$ be the necessarily unique edge connecting the pendant star to the rest of the graph. Associate $e_0 \sim [0,\ell_0]$, where $0$ corresponds to the central vertex of ${\PS}$, and write $\psi_1|_{e_0}(x) = c_0 \sin(k_D(x+\theta_{D}))$ and $\phi_1|_{e_0}(x) = c_0\sin(k_N(x+\theta_{N}))$. The continuity conditions between any leaf and $e_0$ at the central vertex then become
    \begin{equation}
        \label{eq:continuity-condition-dirichlet}
        c_0\sin(k_D\theta_D) = c\sin(k_D\ell)
    \end{equation}
    and
    \begin{equation}
        \label{eq:continuity-condition-neumann}
        c_0\sin(k_N\theta_N) = c\sin(k_N \ell).
    \end{equation}
    Dividing \eqref{eq:continuity-condition-dirichlet} by \eqref{eq:continuity-condition-neumann}, we arrive at
    \begin{equation}
    \label{eq:continuity-quotient}
        \frac{\sin(k_D\theta_D)}{\sin(k_N\theta_N)} = \frac{\sin(k_D\ell)}{\sin(k_N \ell)} \eqqcolon C > 0.
    \end{equation}
    Our goal will be to show that $C = 1$.

    Consider the Kirchhoff condition at the central vertex, which gives us the following two equations:
    \begin{equation}
    \label{eq:kirchhoff-condition-dirichlet}
        nc\cos(k_D\ell) = c_0\cos(k_D\theta_D)
    \end{equation}
    and
    \begin{equation}
    \label{eq:kirchhoff-condition-neumann}
        nc \cos(k_N\ell) = c_0\cos(k_N\theta_N).
    \end{equation}
    Dividing \eqref{eq:kirchhoff-condition-dirichlet} by \eqref{eq:kirchhoff-condition-neumann}, we obtain
    \begin{displaymath}
        \frac{\cos(k_D\ell)}{\cos(k_N\ell)} = \frac{\cos(k_D\theta_D)}{\cos(k_N\theta_N)}.
    \end{displaymath}
    Squaring both sides and using \eqref{eq:continuity-quotient}, we can rewrite the above as
    \begin{displaymath}
        \frac{1-C^2\sin^2(k_N\ell)}{1-\sin^2(k_N\ell)}=\frac{1-C^2\sin^2(k_N\theta_N)}{1-\sin^2(k_N\theta_N)}.
    \end{displaymath}
    
    Suppose now that $C^2 \neq 1$. Then, by the injectivity of the function $x \mapsto \frac{1-C^2x^2}{1-x^2}$ on the interval $[0,1]$, we conclude that $\sin(k_N \ell) = \sin(k_N \theta_N)$. Using \eqref{eq:continuity-condition-neumann}, it then follows that $c = c_0$, which in turn implies that $\sin(k_D\ell) = \sin(k_D\theta_D)$ by \eqref{eq:continuity-condition-dirichlet}.

    The above equalities imply that $\cos(k_D\ell) = \pm \cos(k_D\theta_D)$ and $\cos(k_N\ell) = \pm \cos(k_N\theta_N)$. Substituting these equalities back into \eqref{eq:kirchhoff-condition-dirichlet} and \eqref{eq:kirchhoff-condition-neumann}, we arrive at
    \begin{displaymath}
        (n\pm 1)\cos(k_D\ell) = 0
    \end{displaymath}
    and
    \begin{displaymath}
        (n\pm 1)\cos(k_N\ell) = 0.
    \end{displaymath}
    Since $n \geq 2$ and $0 < \ell < \frac{\pi}{k_D} < \frac{\pi}{k_N}$, the above equations would imply that $k_D\ell$ and $k_N\ell$ are both equal to the first positive zero of the cosine, that is, $k_D = k_N = \frac{\pi}{2\ell}$, an immediate contradiction.

    It follows that $C^2 = 1$, which implies $C = 1$ since $C$ is positive. This in turn implies that $\sin(k_D\ell) = \sin(k_N\ell)$ which has a unique solution given by $k_D\ell = \pi-k_N\ell$ which can be rewritten as $k_D+k_N = \frac{\pi}{\ell}$.
\end{proof}

\begin{proof}[Proof of Lemma~\ref{lem:interval-star-dirichlet-robin-comparison}]
    First, we identify $I \sim [0,L]$ for some $L > 0$, where $0$ corresponds to the Dirichlet vertex and $L$ to the Robin vertex with parameter $\alpha$. The first eigenfunction $\psi_1$ can then be written as $\psi_1(x) = c\sin(k_\alpha^Ix)$ and so the Robin condition reduces to
    \begin{displaymath}
        k_\alpha \cos(k_\alpha L) +\alpha\sin(k_\alpha L) = 0,
    \end{displaymath}
    which, by using the identity $\text{arccot}(-x) = \pi-\text{arccot}(x)$, can be rewritten as
    \begin{equation}
    \label{eq:robin-condition-interval}
        L = \frac{\pi-\text{arccot}\left(\frac{\alpha}{k_\alpha}\right)}{k_\alpha}.
    \end{equation}
    We do the same for the star ${S}$, identifying the edge with the Robin condition $e_0 \sim [0,r]$ for some $r > 0$, where $x = 0$ corresponds to the Robin vertex.
    On that edge, the first eigenfunction $\phi_1$ can be written as $\phi_1|_{e_0}(x) = c_0\sin(k_\alpha^{S}x+\theta)$, where $\theta$ must satisfy:
    \begin{displaymath}
        -k_\alpha^{S} \cos(\theta)+\alpha\sin(\theta) = 0,
    \end{displaymath}
    which can be rewritten as
    \begin{equation}
    \label{eq:robin-condition-star}
        \theta = \text{arccot}\left(\frac{\alpha}{k_\alpha^{S}} \right).
    \end{equation}
    Denoting by $\theta_N = \text{arccot}\left(\frac{\alpha_N}{k_{\alpha_N}^{S}} \right)$ and $\theta_D = \text{arccot}\left(\frac{\alpha_D}{k_{\alpha_D}^{S}} \right)$, we define
    \begin{equation}
    \label{eq:length-choices}
    \ell_N = \frac{\pi-\theta_N}{k_{\alpha_N}^{S}} \quad \text{and} \quad
    \ell_D = \frac{\pi-\theta_D}{k_{\alpha_D}^{S}}.
    \end{equation}
    
    From equation \eqref{eq:robin-condition-interval}, it follows that an interval  $I$ with length $\ell_D$ and Robin parameter $\alpha_D$ has first eigenvalue $\lambda_1^{S}(\alpha_D)$, i.e, if we take $L = \ell_D$, then $k_{\alpha_D}^I = k_{\alpha_D}^{S}$ (analogously for $\alpha_N$ if we take $L = \ell_N$).
    
    To show the desired inequality, we need to show that for this choice of $L = \ell_D$, we have $k_{\alpha_N}^I < k_{\alpha_N}^{S}$. We claim that it suffices to show $\ell_N < \ell_D$.
    
    Indeed, note that, for $\alpha$ fixed, whenever $0 < k_\alpha L < \pi$, the functions $L \mapsto k_\alpha \cot(k_\alpha L)+\alpha$ and $k_\alpha \mapsto k_\alpha \cot(k_\alpha L)+\alpha$ are strictly decreasing. Hence if $\ell_N < \ell_D$, it will follow from the choice of $\ell_N$ \eqref{eq:length-choices} with the identity $-\cot(x) = \cot(\pi-x)$ that
    \begin{displaymath}
        k_{\alpha_N}^{S}\cot(k_{\alpha_N}^{S}\ell_D)+\alpha_N < 0;
    \end{displaymath}
    and since, by the definition of $k_{\alpha_N}^I$,
    \begin{displaymath}
        k_{\alpha_N}^I\cot(k_{\alpha_N}^I\ell_D)+\alpha_N = 0,
    \end{displaymath}
    we can then conclude that $k_{\alpha_N}^I < k_{\alpha_N}^{S}$, which is equivalent to \eqref{eq:interval-star-dirichlet-robin-comparison}.

    Hence, we need to show that
    \begin{equation}
    \label{eq:first-length-inequality}
        \ell_N = \frac{\pi-\theta_N}{k_{\alpha_N}^{S}}<\frac{\pi-\theta_D}{k_{\alpha_D}^{S}} = \ell_D.
    \end{equation}
    First, identifying the Dirichlet leaves of the star ${S}$, with the interval $[0,\ell]$ where $0$ corresponds to the Dirichlet vertex, the Kirchhoff condition at the central vertex can be written as
    \begin{equation}
    \label{eq:robin-kirchhoff-neumann}
        n\cot(k_{\alpha_N}^{S}\ell)+\cot(k_{\alpha_N}^{S}r+\theta_N) = 0
    \end{equation}
    and
    \begin{equation}
    \label{eq:robin-kirchhoff-dirichlet}
        n\cot(k_{\alpha_D}^{S}\ell)+\cot(k_{\alpha_D}^{S}r+\theta_D) = 0.
    \end{equation}
    Since, by hypothesis, $k_{\alpha_D}^{S}+k_{\alpha_N}^{S} = \frac{\pi}{\ell}$, we have that $\cot(k_{\alpha_N}^{S}\ell) = -\cot(k_{\alpha_D}^{S}\ell)$. Therefore, summing \eqref{eq:robin-kirchhoff-neumann} and \eqref{eq:robin-kirchhoff-dirichlet}, we get that
    \begin{displaymath}
        \cot(k_{\alpha_D}^{S}r+\theta_D) + \cot(k_{\alpha_N}^{S}r+\theta_N) = 0.
    \end{displaymath}
    Since $k_{\alpha_D}^{S} > k_{\alpha_N}^{S}$, we have that $k_{\alpha_D}^{S} > \frac{\pi}{2\ell} > k_{\alpha_N}^{S}$, which implies that $\cot(k_{\alpha_N}^{S}\ell) >  0$ and $\cot(k_{\alpha_D}^{S}\ell) < 0$. Using the above equations, this in turn implies that $k_{\alpha_N}^{S}r+\theta_N > \frac{\pi}{2}$ and $k_{\alpha_D}^{S}r+\theta_D < \frac{\pi}{2}$. Hence, the only possible solution to the above equation is
    \begin{displaymath}
        k_{\alpha_D}^{S}r+\theta_D = \pi - (k_{\alpha_N}^{S}r+\theta_N),
    \end{displaymath}
    from which it follows that
    \begin{displaymath}
        \theta_N = \pi\left(1-\frac{r}{\ell}\right)-\theta_D.
    \end{displaymath}
    Using this, $\eqref{eq:first-length-inequality}$ reduces to
    \begin{equation}
    \label{eq:second-length-inequality}
        \frac{\pi-\theta_D}{k_{\alpha_D}^{S}} > \frac{\pi\frac{r}{\ell}+\theta_D}{\frac{\pi}{\ell}-k_{\alpha_D}^{S}}.
    \end{equation}
    To further reduce the number of unknowns in the above inequality, using $\eqref{eq:robin-kirchhoff-dirichlet}$, we can write $r$ explicitly as a function of $k_{\alpha_D}^{S},\theta_D$ and $\ell$:
    \begin{displaymath}
        r = \frac{\pi-\text{arccot}(n\cot(k_{\alpha_D}^{S}\ell))-\theta_D}{k_{\alpha_D}^{S}},
    \end{displaymath}
    which is valid since $k_{\alpha_D}^{S}\ell \in \left(\frac{\pi}{2},\pi\right)$.

    Substituting this into $\eqref{eq:second-length-inequality}$, we obtain
    \begin{displaymath}
         \frac{\pi-\theta_D}{k_{\alpha_D}^{S}} - \frac{\frac{\pi}{\ell}\left( \frac{\pi-\text{arccot}(n\cot(k_{\alpha_D}^{S}\ell))-\theta_D}{k_{\alpha_D}^{S}}\right)+\theta_D}{\frac{\pi}{\ell}-k_{\alpha_D}^{S}} > 0,
    \end{displaymath}
    which can be further simplified to
    \begin{displaymath}
        \frac{\pi\left(\frac{1}{\ell}\text{arccot}(n\cot(k_{\alpha_D}^{S} \ell))-k_{\alpha_D}^{S} \right)}{k_{\alpha_D}^{S}\left(\frac{\pi}{\ell}-k_{\alpha_D}^{S}\right)} > 0.
    \end{displaymath}
    It suffices then to show that, for any $\ell > 0$ and $n \geq 2$, the function $f(x) = \frac{1}{\ell}\text{arccot}(n\cot(x \ell))-x$ is strictly positive on the interval $\left(\frac{\pi}{2\ell},\frac{\pi}{\ell}\right)$. But this is easily shown to be true since $f\left(\frac{\pi}{2\ell}^+\right) = f\left(\frac{\pi}{\ell}^-\right) = 0$ and its second derivative $f''(x) = \frac{2\ell n(n^2-1)\cot(x \ell)\csc^2(x \ell)}{(n^2\cot^2(x \ell)+1)^2}$ is strictly negative on that interval.
\end{proof}

\bibliographystyle{plain}

\end{document}